\documentclass[a4paper]{article}
\usepackage{my_package}

\usepackage{amsfonts}
\usepackage{hyperref}
\usepackage{glossaries-extra}

\newcommand{\notation}[3]{\newglossaryentry{#1}{name={#2}, description={#3}, category=symbol }}

 \notation{sites}{$u,v$}{vertices}
 \notation{lap-rest}{$\Delta_{\bm{\beta}}$}{discrete laplacian on a graph with edge weights ${\bm{\beta}}$. Be ultra-careful with this definition}
 \notation{lap}{$\Delta$}{discrete laplacian}
 \notation{lap-lambda}{$\Delta_\Lambda$}{discrete laplacian restricted to $\Lambda$}
 \notation{lap-g}{$\Delta^w$}{discrete laplacian, with the wired boundary}
\notation{green}{$G_\Lambda$}{Green function of $\Delta_\Lambda$}
 \notation{graph}{$G=(\Lambda,\,\mathcal{E}(\Lambda)$}{Graph with vertices and edges}
 \notation{coordinate}{$\mathfrak E$}{all coordinate vectors $e_1,\,\ldots,\,e_{2d}$}
\notation{fer-gff}{$\psi,\bpsi$}{Discrete fermionic GFF}
\notation{expectationferm}{$\Efi{\cdot}_{\Lambda}/\efi{\cdot}_{\Lambda}$}{Expectation for  fermionic, with wired boundary, with/without normalization, $\mathbf{i}$ stands for the boundary condition: $\mathbf{w}$ for wired, $\mathbf{p}$ pinned at $g$, $\mathbf{0}$ for Dirichlet. }
\notation{partitionfunc}{$\zfi$}{Partition function associated to $\efi{\cdot}$}
\notation{fieldX}{$X_v$}{$\frac{1}{2d} \sum_{i =1}^2d (\psi_{v+e_i}-\psi_{v})(\bpsi_{v+e_i}-\bpsi_{v})$ }
\notation{fieldY}{$Y_v$}{$\prod_{i =1}^{2d} \left(1- (\psi_{v+e_i}-\psi_{v})(\bpsi_{v+e_i}-\bpsi_{v})\right)$ }
\notation{fieldZ}{$Z_v$}{$X_v Y_v$ }
\notation{zetae}{$\zeta_e$}{$(\psi_{e^{+}}-\psi_{e^{-}})(\bpsi_{e^{+}}-\bpsi_{e^{-}})$ }
\notation{zetaE}{$\zeta_E$}{$\prod_{e \in E} \zeta_e $ }

\title{Correlations in uniform spanning trees: a fermionic approach}
\date{\today}
\author{Alan Rapoport}
\affil{\hspace{0.9cm} Utrecht University, Budapestlaan 6, 3584 CD Utrecht, The Netherlands \newline
\url{a.rapoport@uu.nl}}

\begin{document}

\maketitle

\begin{abstract}
In the present paper we establish a clear correspondence between probabilities of certain edges belonging to a realization of the \emph{uniform spanning tree} (UST), and the states of a \emph{fermionic Gaussian free field}. Namely, we express the probabilities of given edges belonging or not to the UST in terms of fermionic Gaussian expectations. This allows us to explicitly calculate joint probability mass functions of the degree of the UST on a general finite graph, as well as obtain their scaling limits for certain regular lattices.
\end{abstract}

%\input{introduction.tex}

% \paragraph{Acknowledgements.} We thank Roland Bauerschmidt, Tyler Helmuth and Dirk Schurricht for many valuable and inspiring discussions. AC acknowledges the hospitality of Utrecht University and AR acknowledges the hospitality of UCL.

\section{Introduction}

In~\citet{cipriani2023fermionic} the authors study the joint moments of the so-called \emph{height-one field} of the \emph{Abelian sandpile model} (ASM), by means of a construction of a local field with \emph{fermionic} variables on a graph. This was achieved given the fact that the height-one field of the ASM at stationarity can be put into correspondence with certain realizations of the \emph{uniform spanning tree} (UST) (\citet{dhar1991,jarai,durre}). By doing so, the authors also managed to obtain closed-form expressions of the joint moments of the degree field of the UST. In the present paper we build up on those techniques to obtain, among other results, a closed-form expression for the probability mass function of the UST.

Our first observation is a general recipe to calculate probabilities of given edges to be or not to be in the UST in terms of fermionic variables, which is the result given in Proposition~\ref{prop:yes_no} in Section~\ref{sec:fermions}. Namely, for any finite graph $\mathcal G = (\Lambda,\,E)$ and directed edges $\{f_i\}_i$, $\{g_j\}_j$ with tail points $\{v_i\}_i$, $\{w_j\}_j$ respectively,
\[
    \P\big(\{f_i\}_i \in \mathrm{UST},\, \{g_j\}_j \notin \mathrm{UST}\big) = \Ef{\prod_{i}\nabla_{f_i}\overline\psi(v_i) \nabla_{f_i}\psi(v_i) \prod_{j}\left(1- \nabla_{g_j}\overline\psi(w_j) \nabla_{g_j}\psi(w_j)\right)},
\]
where $\overline\psi$ and $\psi$ are generators of a Grassmannian algebra and $\Ef{\cdot}$ is the so-called \emph{fermionic Gaussian free field state} (fGFF) (roughly, expected values under the fGFF measure). Precise definitions will be given in Section~\ref{sec:definitions}, but for the moment we stress that these variables satisfy anti-commuting relations as
\[
    \psi(v_i)\psi(v_j) = -\psi(v_j)\psi(v_i),\quad \forall\, i,\,j,
\]
and the fGFF measure is a Gaussian measure on these variables.

On the other hand, there is a well-known connection between these UST probabilities and determinants of the \emph{transfer-current matrix} $M$, which was originally studied to model electric networks. In this context, if $\mathcal G$ is considered as a network where each edge represents a conductance equal to $1$, for any two edges $e$ and $f$ the value of $M(e,\,f)$ is the current measured through $f$ when a battery imposes a unit current through $e$. These values can also be related to local times of a random walk on $\mathcal G$, so $M$ can also be expressed in terms of gradients of the Green's function of the graph in question (see e.g.~\citet{kassel2015transfer}). With this ingredient, the aforementioned fermionic expected values can be written in terms of determinants of $M$.

Afterwards, for $v\in\Lambda$ we can define the fields
\[
    X_v^{(k_v)}\coloneqq \sum_{\mcE\subseteq E_v:\, |\mcE|=k_v} \prod_{e\in\mcE} \nabla_e\overline\psi(v) \nabla_e\psi(v)
\]
for $k_v \in \{1,\,\ldots,\,\deg_{\mathcal G}(v)\}$, being $\deg_{\mathcal G}(v)$ the degree of $v$ on the graph $\mathcal G$, $E_v$ the edges incident to $v$, and
\[
    Y_v\coloneqq \prod_{e\in E_v}\left(1-\nabla_e\overline\psi(v) \nabla_e\psi(v)\right).
\]
With these fields we obtain the joint probability mass functions of the degree field $D_v$ of the UST as
\[
    \P\left(D_v = k_v , \, v\in V\right) = \Ef{\prod_{v\in V} X_v^{(k_v)} Y_v} ,
\]
establishing a clear connection between the fermionic formalism and the UST. This is the result of Theorem~\ref{thm:degree_fgff}. We highlight that fermionic variables have already been used to study problems of random trees, as in~\citet{Caracciolohyperforest} and \citet{bauerschmidt-swan}.

By means of the transfer-current matrix, this result can be further expanded to yield an explicit expression of the joint moments of the fields $\big(X_v^{(k_v)} Y_v\big)_v$ in terms of the Green's function of the graph $\mathcal G$, given in Theorem~\ref{thm:cum_discrete}. To the best of the author knowledge, in the literature there is no full general expression for the exact distribution of the degree field of a UST on a general graph. This can be applied, for example, to calculate the probability of a vertex on the complete graph $K_n$ to have degree $k$, for any $n\geq1$. Taking $n\to\infty$, we show that the degree variable behaves as a Poisson variable plus $1$, a result which was already known (\citet{aldous,pemantle}), but it comes in a more straight-forward manner with our approach since we have an explicit expression of the probability mass function of the degree variable for any $n\geq 1$ at any given point.

Finally, if we take a bounded subset $U\subset\R^d$ and restrict ourselves to a finite subset of regular lattices $\mathbf L$ like $\Z^d$ or the triangular or hexagonal lattices in $d=2$ by taking the intersection $U/\eps\cap\mathbf L$ with $\eps>0$, we can obtain a limiting expression for the joint cumulants of the variables $\big(X_v^{(k_v)} Y_v\big)_v$ when we take the limit of the whole infinite lattice, as
\begin{align}\label{eq:intro_limit_cumulants}
    \tilde\kappa(v_1,\,\dots,\,v_n) &\coloneqq \lim_{\eps\to 0} \eps^{-dn} \kappa\left(\left(X_v^{k_v}\right)^\eps \, Y_v^\eps:\,v\in V\right) 
	\nonumber \\
    &= - \left[\prod_{v\in V}C_\mathbf{L}^{(k_v)}\right] \sum_{\sigma} \sum_{\eta} \prod_{v\in V} \partial_{\eta(v)}^{(1)}\partial_{\eta(\sigma(v))}^{(2)} g_U\left(v,\, \sigma(v)\right),
\end{align}
where $g_U$ is the continuum Green's function on $U$, $\sigma$ are cyclic permutations on $V$, and $\eta$ are the directions of derivation on $\R^d$. Once again, the notation will become more clear after Section~\ref{sec:definitions}. The constants $C_\mathbf{L}^{(k_v)}$ are explicitly calculated in terms of the Green's function values of $\mathbf L$. We observe that the expression for the limiting cumulants are the same for all lattices up to a constant, hinting towards a potential universality property of the system. Unlike in~\citet{cipriani2023fermionic}, the proof of this now more general limiting result is unified for all the lattices considered, which makes the necessary conditions of the lattice more clear for our proof to work. The reader will also observe that expression~\eqref{eq:intro_limit_cumulants} has exactly the same functional form as that of the height-one field of the ASM (\citet{durre,cipriani2023fermionic}), albeit with a different constant in front, meaning that the limiting joint cumulants expressions are affected by the values of $(k_v)_v$ only through $C_{\mathbf L}^{(k_v)}$, but otherwise remain the same.

\paragraph{Structure of the paper.} We begin our paper setting up notation and defining the main objects of interest in Section~\ref{sec:definitions}. Section~\ref{sec:fermions} is devoted to recapitulate on the fermionic formalism used throughout the paper, as well as stating the first general results linking fermionic Gaussian states and UST probabilities. The moments/cumulants, both for a finite graph and the limiting case, are in Section~\ref{sec:cumulants}. At the end of that section we also discuss the case of the complete graph $K_n$ and its limit $n\to\infty$. Finally, Section~\ref{sec:proofs} is devoted to the proofs of the main theorems.

\paragraph{Acknowledgments.} The author warmly thanks Alessandra Cipriani and Wioletta Ruszel for their paramount input and guidance, leading to the proposal and completion of the present article. We also thank Leandro Chiarini for many valuable discussions.

\section{Notation and definitions}\label{sec:definitions}

\paragraph{Lattices, sets and functions}
For the rest of the paper $d$ will denote the dimension of the underlying space we work on.
We will write $|A|$ for the cardinality of a set $A$. For $n\in \N$, let $[n]$ denote the set $\{1,\,\ldots,\,n\}$.

Throughout the paper $\mathbf L$ will denote a lattice. In particular, we will consider what we will call the hypercubic lattice $\Z^d$, the two dimensional triangular lattice $\mathbf T$, and we will also make some remarks on the two dimensional hexagonal lattice $\mathbf H$. Since these lattices are regular, we write $\deg_{\mathbf L}$ for the degree of any vertex, which is $2d$ for $\Z^d$, $6$ for $\mathbf T$, and $3$ for $\mathbf{H}$.

We will denote an oriented edge $f$ on the lattice $\mathbf L$ as the ordered pair $(f^-,\,f^+)$, being $f^-$ the tail and $f^+$ the tip of the edge. Denote $\{e_i\}_{i\in[\deg_{\mathbf L}]}$ the set of edges with tail in the origin. The $e_i$'s define a natural orientation of edges which we will tacitly choose whenever we need oriented edges (for example when defining the matrix $M$ in~\eqref{eq:M}). The opposite vectors will be written as $e_{d+i} \coloneqq -e_i,\,i=1,\,\ldots,\,d$. Furthermore
\[
    \tilde e_i \coloneqq (0,\,\ldots,\,0,\,\underbrace{1}_{i\text{-th position}},\,0,\,\ldots,\,0),\quad i=1,\,\ldots,\,d
\]
denotes the $d$ standard coordinate vectors of $\R^d$.

The collection of all $e_i$, $i\in \{1,\,\ldots,\,\deg_{\mathbf L}\}$, will be called $E_o$, where $o$ is the origin. By abuse of notation but convenient for the paper, if $f=(f^-,\,f^-+e_i)$ for some $i\in\left[\deg_{\mathbf L}\right]$, we denote by $-f$ the edge $(f^-,\,f^- - e_i)$ whenever it exists; that is, the reflection of $f$ over $f^-$. 

Call $A \subseteq \R^d$ a countable set. For every $v \in A$, denote by $E_v$ the set $E_o+v$, and let $E(A) = \bigcup_{v\in A} E_v$. 

Let $U\subseteq \R^d$ and $e \in E_o$. For a function $g:\,U\to\R^d$ differentiable at $x$ we define $\partial_e g(x)$ as the directional derivative of $g$ at $x$ in the direction corresponding to $e$, that is 
\begin{equation*}
	\partial_e g(x) 
=
	\lim_{t \to 0^{+}} \frac{g(x+t  e)-g(x)}{t}.
\end{equation*}
% 
% When $e=e_i$ for $i =1,\,\dots,\,2d$, we simply write $\partial_i f(x)$.
%
Likewise, when we consider a function in two variables $g: \mathbb{R}^d \times \mathbb{R}^d \to \mathbb{R}$, we write then $\partial^{(j)}_e g(\cdot,\cdot)$ to denote the directional derivative in the $j$-th entry, $j=1,\,2$.

\paragraph{Graphs and Green's function} 
As we use the notation $(u,\,v)$ for a directed edge we will use $\{ u,\, v\}$ for the corresponding undirected edge. For a finite (unless stated otherwise) graph $\mathcal{G}=(\Lambda,\, E)$ we denote the degree of a vertex $v$ as $\deg_\mathcal{G}(v) \coloneqq \left|\{u \in \Lambda:\, u \sim v\}\right|$, where $u\sim v$ means that $u$ and $v$ are nearest neighbors.

\begin{definition}[Discrete derivatives]
For a function $g:\,\mathbf L\to \R$, its discrete derivative $\nabla_{e_i} g$ in the direction $i=1,\,\ldots,\,\deg_{\mathbf L}$ is defined as
\[
\nabla_{e_i} g(u)\coloneqq g(u+e_i)-g(u),\quad u\in\mathbf L.
\]
Analogously, for a function $g:\,\mathbf L\times \mathbf L\to\R$ we use the notation $\nabla^{(1)}_{e_i}\nabla^{(2)}_{e_j} g$ to denote the double discrete derivative defined as
\[
\nabla^{(1)}_{e_i}\nabla^{(2)}_{e_j}g(u,\,v)\coloneqq g(u+e_i,\,v+e_j)- g(u+e_i,\,v)-g(u,\,v+e_j)+g(u,\,v),
\]
for $u,\,v\in \mathbf L$, $i,\,j=1,\,\ldots,\,\deg_{\mathbf L}$.
\end{definition}
\begin{definition}[Discrete Laplacian on a graph] 
We define the (unnormalized) {\it discrete Laplacian} on $\mathbf L$ as
\begin{equation}\label{eq:laplacian}
    \Delta(u,\,v) \coloneqq
    \begin{dcases}
    \hfil -\big|\{w\in\mathbf L:\, w\sim u\}\big| & \text{if } u=v,\\
    \hfil 1 & \text{if } u \sim v,\\
    \hfil 0 & \text{otherwise.}
    \end{dcases}
\end{equation}
where $u,\,v \in \mathbf L$ and $u \sim v$ denotes that $u$ and $v$ are nearest neighbors. For any function $g:\,\mathbf L \to \mathcal A$, where $\mathcal A$
is an algebra over $\R$, we define 
\begin{equation}\label{eq:laplacian_on_function}
	\Delta g(u) \coloneqq 
	\sum_{v \in \mathbf L}\Delta(u,\,v) g(v) =
	\sum_{v \sim u} (g(v)-g(u)).
\end{equation}
\end{definition}
Note that we define the function taking values in an algebra because we will apply the Laplacian both on real-valued functions and functions defined on Grassmannian algebras, which will be introduced in Section~\ref{sec:fermions}.

We also introduce $\Delta_\Lambda\coloneqq (\Delta(u,\,v))_{u,\,v\in \Lambda}$, the restriction of $\Delta$ to $ \Lambda$.
Notice that for any lattice function $f$ we have that for all $u \in \Lambda$, 
\begin{equation}\label{eq:Delta_Lam}
\Delta_\Lambda g(u) = \sum_{v \in \Lambda} \Delta(u,v) g(v) = \Delta g_\Lambda (u)
\end{equation}
where $g_\Lambda$ is the lattice function given by $ g_\Lambda(u) \coloneqq g(u) \1_{u\in\Lambda}$.

The exterior boundary of a set $\Lambda$ will be defined by\label{boundary}
% \[
% \partial^{\mathrm{in}}\Lambda \coloneqq \{u\in\Lambda:\,\exists \,v\in\mathbf L\setminus\Lambda:\,u\sim v\},
% \]
% and the outer boundary by
\[
\partial^{\mathrm{ex}}\Lambda \coloneqq \{u\in\mathbf L\setminus\Lambda:\,\exists \,v\in\Lambda:\,u\sim v\}.
\]
% We also consider the interior of $\Lambda$, given by $\Lambda^{\mathrm{in}} \coloneqq \Lambda \setminus \partial^{\mathrm{in}}\Lambda$. The notation $\partial U$ will also be used to denote the boundary of a set $U \subseteq \R^d$.

\begin{definition}[Discrete Green's function]
Let $u \in \Lambda$ be fixed. The Green's function $G_\Lambda(u,\cdot)$ with Dirichlet boundary conditions is defined as the solution of
\[
\begin{dcases}
-\Delta_{\Lambda} G_{\Lambda}(u,\,v) = \delta_u(v) &\text{ if } v\in \Lambda, \\
\hfil G_{\Lambda}(u,\,v) = 0 &\text{ if } v\in \partial^{\mathrm{ex}} \Lambda,
\end{dcases}
\]
where $\Delta_\Lambda$ is defined in \eqref{eq:Delta_Lam}. 
\end{definition}

\begin{definition}[Infinite volume Green's function, {\cite[Sec. 1.5-1.6]{lawler2013intersections}}]\label{def:Green-function-discrete}
    With a slight abuse of notation we denote by $G_0(\cdot,\cdot)$ two objects in different dimensions:
    \begin{itemize}
    \item $d\ge 3$ (only for $\Z^d$): $G_0$ is the solution of
    \[
    \begin{dcases}
        -\Delta G_0(u,\cdot)=\delta_u(\cdot)\\
        \lim_{\|v\|\to\infty}G_0(u,\,v) = 0
    \end{dcases}, \quad u\in \mathbf L.\\
    \]
	
    \item $d=2$: $G_0$ is given by
    \[
    G_0(u,\,v)= -\frac{1}{\deg_{\mathcal G}(u-v)}a(u-v),\quad u,\,v\in \mathbf L,
    \]
    where $a(\cdot)$ is the potential kernel defined as
    \[
    a(u) = \sum_{n=0}^\infty \big[\P_o(S_n=o)-\P_o(S_n=u)\big], \quad u\in \mathbf L,
    \]
    and $\{S_n\}_{n\geq0}$ is a random walk on the plane staring at the origin and $\P_o$ its probability measure.
    \end{itemize}
\end{definition}

\paragraph{Cumulants} We now give a brief recap of the definition of cumulants and joint cumulants for random variables. Let $n\in\N$ and $\mathbf X=(X_{i})_{i=1}^n$ be a vector of real-valued random variables, each of which has all finite moments.

\begin{definition}[Joint cumulants of random vector]\label{def:cum_mult}
The cumulant generating function $K(\mathbf t)$ of $\mathbf X$ for $\mathbf t=(t_1,\,\dots,\,t_n)\in \R^n$ is defined as 
\[
K(\mathbf t) \coloneqq \log \left ( \mathbb{E}\big[e^{ \mathbf t\cdot \mathbf X}\big]\right ) = \sum_{\mathbf m\in\N^n} \kappa_{\mathbf m}(\mathbf X) \prod_{j=1}^n \frac{t_j^{m_j}}{m_j!} \ ,
\]
where $\mathbf t\cdot \mathbf X$ denotes the scalar product in $\R^n$, $\mathbf m=(m_1,\,\dots,\,m_n)\in\N^n$ is a multi-index with $n$ components, and
\[
\kappa_{\mathbf m}(\mathbf X)=\frac{\partial^{|m|}}{\partial t_1^{m_1}\cdots \partial t_n^{m_n}}K(\mathbf t)\Big|_{t_1=\ldots=t_n=0} \ ,
\]
being $|m|=m_1+\cdots+m_n$. The joint cumulant of the components of $\mathbf X$ can be defined as a Taylor coefficient of $K(t_1,\,\ldots,\,t_n)$ for $\mathbf m=(1,\,\ldots,\,1)$; in other words
\[
\kappa(X_1,\,\ldots,\,X_n)=\frac{\partial^n}{\partial t_1\cdots \partial t_n} K(\mathbf t)\Big|_{t_1=\ldots=t_n=0} \ .
\]
In particular, for any $A\subseteq [n]$, the joint cumulant $\kappa(X_i:\, i\in A)$ of $\mathbf X$ can be computed as 
\[
\kappa(X_i:\, i\in A) = \sum_{\pi \in \Pi(A)} (|\pi|-1)! (-1)^{|\pi|-1} \prod_{B\in \pi} \mathbb{E} \left[\prod_{i\in B} X_i \right] \ ,
\]
with $|\pi|$ the cardinality of $\pi$.
\end{definition}
Let us remark that, by some straightforward combinatorics, it follows from the previous definition that
\begin{equation}\label{eq:mom_to_cum}
\mathbb{E} \left[\prod_{i\in A} X_i \right] = \sum_{\pi \in \Pi(A)} \prod_{B\in \pi} \kappa(X_i:\, i\in B) \ .
\end{equation}
If $A=\{i,\,j\}$, $i,\,j\in[n]$, then the joint cumulant $\kappa(X_i,\,X_j)$ is the covariance between $X_{i}$ and $X_{j}$. We stress that, for a real-valued random variable $X$, one has the equality
\[
\kappa(\underbrace{X,\,\ldots,\,X}_{n\text{ times}})=\kappa_n(X),\quad n\in \N,
\]
which we call the \emph{$n$-th cumulant of $X$}.

\paragraph{Good sets and transfer-current matrix} We need to introduce a technical requirement for the sets we will study in the theorems that follow, that prevent us to choose points that share edges. This requirement can however be circumvented, as we show in Section~\ref{not_good}.

\begin{definition}[Good set]\label{def:goodset}
We call $A \subseteq \Lambda$ a \emph{good set} if it does not contain any nearest neighbors. That is, $||v-u|| > 1$ for any $u,\,v\in A$.
\end{definition}

Finally, we need the notion of the \emph{transfer-current matrix}, a key ingredient in many expressions we obtain in our theorems.

\begin{definition}[Transfer-current matrix]\label{def:M}
We define the \emph{transfer-current matrix} $M_\Lambda$ as
\begin{equation}\label{eq:M}
    M_\Lambda(f,\,g) \coloneqq 
	\nabla_{\eta^*(f)}^{(1)}\nabla_{\eta^*(g)}^{(2)}G_{\Lambda}(f^{-},\,g^-) ,\quad f,\,g \in E(\Lambda),
\end{equation}
where $\eta^*(f)\in E_o$ is the coordinate direction induced by $f\in E(\Lambda)$ on $f^-$ (in the sense that $\eta^*(f)=e_i$ if $f=(f^-,\, f^-+e_i)$). Hereafter, to simplify notation we will omit the dependence of $M_\Lambda$ on $\Lambda$ and simply write $M$.
\end{definition}

\begin{remark}
    As stated in~\citet{lyonsperes}, there is another definition of $M$ in terms of electrical networks, as follows: let $\mathcal G$ represent an electric network with impedance $1$ on each edge. Defining $\phi_{f}(x)$ as the voltage at vertex $x\in\Lambda$ when a battery of $1$ volt is connected between vertices $g^-$ and $g^+$ by removing the resistance on $g$ and setting the voltage at $f^-$ to $0$, $M(f,\,g)$ is given by
    \[
        M = \phi_f(g^+) - \phi_f(g^-) .
    \]
\end{remark}
\section{Fermionic formalism}\label{sec:fermions}

As we will see in this section, working with Grassmann variables (i.e., fermions) allow us to express properties from the UST. We will not give a complete exposition of the subject in this paper; however, the interested reader can resort to \citet{cipriani2023fermionic} for a similar setting to the one used here, or~\cite{abdesselam,Meyer} for a more comprehensive presentation.

\begin{definition}[{\citet[Definition 1]{abdesselam}}]\label{def:GrasAlg}
Let $M\in \N$ and $\xi_1,\,\ldots,\,\xi_M$ be a collection of letters.  Let $\R\left[\xi_{1},\,\ldots,\,\xi_{M}\right]$ be the quotient of the free non-commutative algebra $\R\langle\xi_1,\,\ldots,\,\xi_M\rangle$ by the two-sided ideal generated by the anticommutation relations
\begin{equation}\label{eq:anticommute}
\xi_j \xi_j = - \xi_i \xi_j,
\end{equation}
where $i,\,j \in [M]$. We will denote it by $\Omega^{M}$ and call it the Grassmann algebra in $M$ variables. The $\xi$'s will be referred to as {\it Grassmannian variables} or {\em generators}. Due to anticommutation these variables are called ``fermionic'' (as opposed to commutative or ``bosonic'' variables).
\end{definition}

Notice that, due to the anticommutative property, we have that for any variable Pauli's exclusion principle holds~\citep[Proposition 2]{abdesselam}:
\begin{equation}\label{eq:squares-are-zero}
\xi_i^2 = 0, \quad i\in[M].
\end{equation}
\begin{definition}[Grassmannian--Berezin integration]
The Grassmann--Berezin integral is defined as
\[
\int F d \xi \coloneqq \partial_{\xi_{M}}\partial_{\xi_{M-1}} \cdots \partial_{\xi_2}\partial_{\xi_1}F, \quad F\in \Omega^{M}.
\]
\end{definition}

On the grounds of this definition, for the rest of the paper Grassmannian--Berezin integrals will be denoted by $\left(\prod_{i=1}^M\partial_{\xi_i}\right) F$.

The most important result that we will use in this work is as follows. For a given matrix $A = (A_{i,\,j})_{i \in I_0,\, j \in J_0}$, and $I \subseteq I_0,\, J\subseteq J_0$, such that $|I|=|J|$, we write $\det (A)_{IJ}$ to denote the determinant of the submatrix $(A_{i,\,j})_{i \in I,\, j \in J}$. When $I=J$, we simply write $\det(A)_I$.

\begin{theorem}[Wick’s theorem for ``complex'' fermions]\label{thm:sportiello}
Let $A $ be an ${m\times m}$, $B$ an $ {r\times m}$ and $C$ an $ {m\times r}$ matrix respectively with coefficients in $\R$.
For any sequences of indices $I = 
\{i_1,\,\dots,\,i_r\}$ and $J = \{j_1,\,\dots,\,j_r\}$ in $[m]$ of the same length $r$, if the matrix $A$ is invertible we have
\begin{enumerate}[label=\arabic*.,ref=\arabic*.]
    \item\label{thm_Wick_one} $\displaystyle \left(\prod_{i=1}^m\partial_{\bpsi_i}\partial_{\psi_i}\right) \prod_{\alpha=1}^r  \psi_{i_\alpha}\bpsi_{j_\alpha}\exp\left((\bm{\psi},\,A\bm{\bpsi})\right) = \det(A) \det\left(A^{-\intercal}\right)_{IJ}$,

    \item\label{thm_Wick_two} $\displaystyle \left(\prod_{i=1}^m\partial_{\bpsi_i}\partial_{\psi_i}\right) \prod_{\alpha=1}^r (\psi^T C)_\alpha (B\bpsi)_\alpha\exp\left((\bm{\psi},\,A\bm{\bpsi})\right) = \det(A) \det\left(BA^{-1}C\right)$.
\end{enumerate}
If $|I| \neq |J|$, the integral is $0$.
\end{theorem}

\begin{definition}[Fermionic Gaussian free field]
The normalized fermionic Gaussian free field state is the linear map $\Ef{\cdot}:\, \Omega^{2\Lambda} \to \mathbb{R}$ defined as
\[
\Ef{F} \coloneqq \frac{1}{\det(-\Delta_\Lambda)} \left(\prod_{v\in\Lambda}\partial_{\bpsi_v}\partial_{\psi_v}\right)\exp\left(\inpr{\bm{\psi}}{-\Delta_\Lambda \bm{\bpsi}}\right)F,\quad F\in \Omega^{2\Lambda}.
\] 
\end{definition}

\subsection{Fermions and the Uniform Spanning Tree}

Now that we have the key ingredients to work with fermions, we will see how these objects allow us to study probabilistic behaviors of the edges of a realization of the UST measure.

We will first consider gradients of the generators in the following sense. 

\begin{definition}[Gradient of the generators]
The gradient of the generators in the $i$-th direction is given by
\[
\nabla_{e_i}\psi(v)=\psi_{v+e_i}-\psi_v,\quad\nabla_{e_i}\bpsi(v)=\bpsi_{v+e_i}-\bpsi_v,\quad v\in \Lambda,\,i=1,\,\ldots,\,\deg_\mathcal{G}(v).
\]
\end{definition}
\noindent
% Remember that $e_{d+i} \coloneqq -e_{i}$ for $1 \le i \le d$, as introduced at the beginning of Section~\ref{sec:notation-and-definitions}.

\vspace{3mm}
\noindent
Define $\zeta(e)$ as
\[
    \zeta(e) \coloneqq \nabla_e\psi(e^-)\nabla_e\bpsi(e^+), 
\]
and observe that the elements $\zeta(\cdot)$ are commutative, that is,
\[
    \zeta(a)\zeta(b) = \zeta(b)\zeta(a),\quad \forall \, a,\,b \in E,
\]
but we still have that $\zeta(a)^2 = 0$. These objects are key when analyzing probabilities of edges showing up in the UST, in the sense of the result that follows.

\begin{proposition}\label{prop:yes_no}
    Let the tree $T$ be a realization of the UST measure. For $F,\,G \subseteq E$, $F\cap G = \emptyset$ it holds that
    \[
        \P\left(F \subseteq T,\, G\cap T = \emptyset\right) = \Ef{\prod_{f\in F} \zeta(f) \prod_{g\in G} \left(1 - \zeta(g)\right)} = \det\big(M^{(|F|)}\big),
    \]
    where
    \[
        M^{(|F|)}(i,\,j) = 
        \begin{dcases}
            M(i,\,j) \quad &\text{if } \, i\leq |F|,\\
            -M(i,\,j) &\text{if } \, |F|+1\leq i\leq |F|+|G|,\, i\neq j, \\
            1-M(i,\,j) & \text{if } \, |F|+1\leq i\leq |F|+|G|,\, i=j.
        \end{dcases}
    \]
\end{proposition}

\begin{proof}
Observe that
\[
    \prod_{g\in G}(1-\zeta(g)) = \sum_{\gamma \subseteq G}(-1)^{|\gamma|}\prod_{g\in\gamma}\zeta(g),
\]
so that
\[
    \Ef{\prod_{f\in F} \zeta(f) \prod_{g\in G}(1-\zeta(g))} = \sum_{\gamma\subseteq G}(-1)^{|\gamma|} \Ef{\prod_{g\in\gamma}\zeta(g) \prod_{f\in F}\zeta(f)} = \sum_{\gamma\subseteq G} (-1)^{|\gamma|} \,\P\left(F \subseteq T,\,\gamma\subseteq T\right).
\]
Using the inclusion-exclusion principle, we obtain the first equality. The equality between the first and third members follows from \citet[Theorem 4.3]{pemantle} (noting that there is a typo in their definition of $M^{(|F|)}$).
\end{proof}

\begin{remark}\label{rmk:fermions}
In view of Proposition~\ref{prop:yes_no} we have the following recipe to cook up a field whose expectation matches that of the UST: for each edge $f$ we want in the UST, add a factor $\zeta(f)$, and for each edge $g$ we do not want, add a factor $1-\zeta(g)$. Observe that, once we add an edge $e$ by adding the factor $\zeta(e)$, then adding another factor $1-\zeta(e)$ does nothing. This is easily seen from the fact that
\[
    \Ef{\zeta(e) \left(1-\zeta(e)\right)} = \Ef{\zeta(e)} - \Ef{\zeta(e)^2} = \Ef{\zeta(e)} .
\]
\end{remark}

\subsection{Degree of the Uniform Spanning Tree}\label{subsec:degree_ust}

So far we have seen the relationships between fermionic variables and particular edges on a spanning tree. We will now use those results to study the behaviour of the degree of a realization of the UST measure at given points on the graph.

\begin{remark}\label{rem:direct-vs-not-directed}
So far we have defined edges on graphs to be oriented. However, in the following definitions the orientation play no r\^ole, so we will consider edges as non oriented.
% Throughout this article we will work with directed edges to encode discrete derivatives in observables of interest.
% However, whenever we are referring to graphs, the Laplacian operator and spanning trees we will always think of undirected graphs.
% In fact, one can show that all of the fields ${\bf X^{k}}$ and ${\bf Y}$ defined in Subsection~\ref{subsec:degree_ust} remain the same if one changes the direction of any/all edges.
\end{remark}

Let $\mathcal G = (\Lambda,\,E)$ be any graph. For each $v\in \Lambda$ and $k_v \in \{1,\,\dots,\,\deg_{\mathcal{G}}(v)\}$, we define the field $\mathbf{X^{(k)}} = \big(X_v^{(k_v)}\big)_{v\in V}$ as
\begin{equation}\label{eq:defX}
    X_v^{(k_v)}\coloneqq \sum_{\mcE\subseteq E_v:\, |\mcE|=k_v} \prod_{e\in\mcE} \zeta(e),\quad v\in \Lambda.
\end{equation}
In view of Remark~\ref{rmk:fermions}, this is equivalent to asking that the degree of the UST at a point $v$ is at least $k_v$; that is,
\[
    \sum_{\mcE\subseteq E_v:\, |\mcE|=k_v} \prod_{e\in\mcE} \1_{\{e\in T\}} ,\quad v\in \Lambda.
\]
% so that\todo{Explain this because I don't remember why this is and what it implies}
% \[
%     \Ef{X_v^{k_v}} = \E\binom{D_v}{k_v}.
% \]
%
If $k_v=1$ for all $v$, this is just the field $(X_v)_v$ defined in \citet{cipriani2023fermionic}. Observe also that, because of the nilpotency property of fermions,
\[
    X^{(k_v)}_v = (X_v)^{k_v},
\]
so we will sometimes indistinctly denote it as $X_v^{k_v}$ The same applies for $\mathbf X^{(\mathbf k)}$ written as $\mathbf{X^k}$. We will also need auxiliary Grassmannian observables ${\bf Y}=(Y_v)_{v\in V}$ given by
\begin{equation}\label{eq:defY}
    Y_v\coloneqq \prod_{e\in E_v}\left(1-\zeta(e)\right),\quad v\in \Lambda.
\end{equation}
Define the degree field of the UST $(D_v)_{v\in\Lambda}$ as
\begin{equation}\label{eq:degree}
    D_v \coloneqq \sum_{e\in E_v}\1_{\{e\in T\}},
\end{equation}
which is ``equal'' (in the sense of its finite-dimensional distributions) to $\big(X_v\big)_v$, as it was seen in \citet{cipriani2023fermionic}. More precisely, for $V\subseteq\Lambda$ a good set (neighboring points will be dealt with in Section~\ref{not_good}),
\[
    \E\left[\prod_{v\in V}D_v\right] = \Ef{\prod_{v\in V}X_v}.
\]
For $k_v\in\{1,\,\dots,\,\deg_{\mathcal G}(v)\}$, define the degree-$k_v$ field as
\[
    \delta^{(k_v)}_v = \1_{\{D_v = k_v\}}
\]
As a consequence of Proposition~\ref{prop:yes_no}, we can express the probability of the degree being a certain value of the UST at different not neighboring points, as in theorem that follows.

\begin{theorem}\label{thm:degree_fgff}
Let $V\subset\Lambda$ be a good set. For any $k_v\in\{1,\,\dots,\,\deg_{\mathcal G}(v)\}$, with $v\in V$, it holds that
\begin{equation}
\P\left(D_v = k_v , \, v\in V\right) = \E\left[\prod_{v\in V}\delta^{(k_v)}_v\right] = \Ef{\prod_{v\in V} X_v^{k_v} Y_v} .
\end{equation}
\end{theorem}

Note that this is a generalization of~\citet[Theorem 3.1]{cipriani2023fermionic}, where we obtain the same result for $k_v=1$ for all $v\in V$, even though in that case our main focus was the height-one field of the Abelian sandpile model.

\begin{remark}
Observe that points in $V$ need to be different. In fact, for $v\in V$,
\[
    \E\left[D_v^2\right] \neq \Ef{X_v^2},
\]
and of course neither does it hold for larger powers. This is because the square of an indicator function (see~\eqref{eq:degree}) is the same indicator, whereas the square of $\zeta(e)$, $e\in E$ (see~\eqref{eq:defX}), is $0$. However, using Proposition~\ref{prop:yes_no} we observe that
\[
    \Ef{X_v(X_v+1)} = \Ef{X_v^2} + \Ef{X_v} = \sum_{\substack{e,\,f\in E_v\\e\neq f}}\det{(M)}_{e,\,f} + \sum_{e\in E_v} M(e,\,e) = \E\left[D_v^2\right] .
\]
Following the same reasoning,
\[
    \E\big[D_v^m\big] = \sum_{i\in[m]} a_i \Ef{X_v^i},
\]
where the coefficients $a_i$ correspond to a modification of the binomial coefficients. More precisely, for $i=1,\,\dots,\,\floor{m/2}$
\[
    a_i = \binom{m}{i-1},
\]
while for $i=\floor{m/2} + 1,\,\dots,\,m$
\[
    a_i = \binom{m}{i}.
\]
We could also find the reverse expression, that is, $\Ef{X_v^m}$ as a function of $\E[D_v^i]$, $i\in[m]$. We can use the results on~\citet[Section 5.2]{pemantle} to obtain
\[
    \Ef{X_v^m} = m!\, \E\left[\binom{D_v}{m}\right] = \E\left[\prod_{i=0}^{m-1}(D_v-i)\right]
\]
for any $m\in\N$.
\end{remark}
\section{Cumulants of the UST degree}\label{sec:cumulants}

We will now study the cumulants (related to moments, as seen in Section~\ref{sec:definitions}) of the fields ${\bf X^k Y}$ on an arbitrary graph, and then obtain limiting expressions for some particular lattices. The next theorem is a generalization of~\citet[Theorem 3.5]{cipriani2023fermionic} when $k_v=1$ for all points $v\in V$.

Let $U$ be connected, bounded subset of $\R^d$ with smooth boundary, and define $U_\eps \coloneqq U/\eps \cap \mathbf L$ for $\eps>0$.
For any $v\in U$, let $v_\eps$ be the discrete approximation of $v$ in $U_\eps$; that is, $v_\eps \coloneqq \floor{v/\eps}$.
Define $g_U$ as the continuum harmonic Green's function on $U$ with $0$-boundary conditions outside $U$.
We write $(X_v^{k_v})^\eps$, $(\bf{X^k})^\eps$ and $Y_v^\eps$, $\bf{Y}^\eps$ to emphasize the dependence of $v$ on $\eps$ whenever $v$ belongs to $U_\eps$. Cyclic permutations without fixed points of a finite set $A$ are denoted as $S_\text{cycl}(A)$. We will also need the so called \emph{connected permutations} $S_{\co}(A)$ of $A$, the definition of which we defer to Subsection~\ref{subsec:graphs-and-permutations}, where we prove the theorems.

\begin{theorem}[Cumulants of ${\mathbf{X^k} \, \mathbf Y}$ on a graph]\label{thm:cum_discrete}
Let $\mathcal G = (\Lambda,\,E)$ be any graph. Let $n\geq 1$, $V\coloneqq \left\{v_1,\,\dots,\,v_n\right\}\subseteq \Lambda^{\mathrm{in}}$ be a good set, with $v_i\neq v_j$ for all $i\neq j$. For a set of edges $\mcE\subseteq E$ and $v\in V$ denote $\mcE_v\coloneqq \{f\in\mcE:\,f^- = v\}\subseteq E_v$. The $n$-th joint cumulants of the fields $\big(X_v^{k_v} Y_v\big)_{v\in V}$ are given by
\begin{equation}\label{eq:thm:maincum3}
    \kappa\left(X_v^{k_v} Y_v:\,v\in V\right) 
	= 
	(-1)^{\sum_v k_v}
	\sum_{\mcE\subseteq E:\, |\mcE_v|\ge k_v \, \forall v} K(\mcE) 
	\sum_{\tau\in S_{\co}(\mcE)} \sign(\tau) \prod_{f\in \mcE} M\left(f,\,\tau(f)\right)
\end{equation}
where
\[
    K(\mcE)\coloneqq\prod_{v\in V}K(\mcE_v), \quad K(\mcE_v)\coloneqq (-1)^{|\mcE_v|}\binom{|\mcE_{v}|}{k_v},
\]
$M=M_{E(V)}$, and $k_v \in\N$ for all $v\in V$.
\end{theorem}

\begin{remark}
The reader might be wondering why we work with cumulants instead of moments in this case, which in view of Theorem~\ref{thm:degree_fgff} it seems to only introduce complications. The reason for this is that cumulants are independent of the mean, which allows us to obtain a limiting result in the next theorem without the need of renormalizing.
\end{remark}

Let $\alpha\in \{0,\,\dots,\,p - 1\}$, where $p$ is the number of edges contained in any two dimensional plane generated by any two edges incident on any $v\in V$; that is, $4$ for the hypercubic lattice in $d$ dimensions, $6$ for the triangular lattice and $3$ for the hexagonal one. Let $\gamma_\alpha \coloneqq \cos \left(2\pi\alpha/p\right)$. This next theorem is a generalization of~\citet[Theorems 3.6 and 5.1.2]{cipriani2023fermionic} when $k_v=1$ for all $v\in V$. We unify their statements and proofs in one theorem.

\begin{theorem}[Scaling limit of the cumulants of $\mathbf{X^{k} \, Y}$]\label{thm:cum_cont}
Let $n\ge 2$, $V\coloneqq \left\{v_1,\,\dots,\,v_n\right\}\subseteq U$ be a good set such that $\dist(V,\,\partial U)>0$, and $\mathbf L$ the lattice $\Z^d$ or $\mathbf T$. Let $\left(\left(X_v^{k_v}\right)^\eps Y_v^\eps\right)_v$ be defined on $U_\eps = U/\eps \cap \mathbf L$. If $v_i \neq v_j$ for all $i\neq j$, then
\begin{align}\label{eq:cum_limit3}
    \tilde\kappa(v_1,\,\dots,\,v_n) &\coloneqq \lim_{\eps\to 0} \eps^{-dn} \kappa\left(\left(X_v^{k_v}\right)^\eps \, Y_v^\eps:\,v\in V\right) 
	\\ &= - \left[\prod_{v\in V}C_\mathbf{L}^{(k_v)}\right] \sum_{\sigma\in S_{\cycl}(V)} \sum_{\eta:\,V\to \{\tilde e_1,\,\ldots,\,\tilde e_d\}} \prod_{v\in V} \partial_{\eta(v)}^{(1)}\partial_{\eta(\sigma(v))}^{(2)} g_U\left(v,\, \sigma(v)\right),
\end{align}
where the constants $C_\mathbf{L}^{(k_v)}$ are given by
\begin{align}\label{eq:const_C}
    C_\mathbf{L}^{(k_v)} = (-1)^{k_v+1} \, c_{\mathbf L} \sum_{\substack{\mcE\in E_o:\, \mcE\ni e_1\\|\mcE|\geq k_v}}(-1)^{|\mcE_v|}\binom{|\mcE|}{k_v} \left[ \det\left(\overline M \right)_{\mcE\setminus{\{e_1\}}} - \sum_{\alpha=1}^{p-1} \gamma_\alpha \1_{\{e_{1+\alpha}\in \mcE\}}  \det\big(\overline M^\alpha\big)_{\mcE\setminus{\{e_1\}}}\right],
\end{align}
where $c_{\Z^d} = 2$ for all $d\geq2$, $c_{\mathbf T} = 3$, and for any $f,\,g\in E_v$
\[
\overline M(f,\,g) =
\nabla_{\eta^*(f)}^{(1)}\nabla_{\eta^*(g)}^{(2)}G_{0}(f^{-},\,g^-)
\]
and
\begin{equation}\label{eq:Malpha}
\overline M^\alpha(f,\,g) =
\begin{cases}
\overline{M}(e_1,\,g) & \text{ if } f = e_{1+\alpha},\\
\overline{M}(f,\,g) & \text{ if } f \neq e_{1+\alpha}.
\end{cases} 
\end{equation}
\end{theorem}

\begin{remark}
As we will see in the proof, the same techniques are immediately generalizable to the hexagonal lattice; that is, for $\mathbf L = \mathbf H$. However, that case requires more care, since we have to account for the two types of vertices in that lattice. We believe an adaptation of the proof to that case only adds obscurity to the matter, but nonetheless it can still be done, yielding the same expression with $p=3$ and $c_{\mathbf H} = 3/2$.
\end{remark}

\begin{remark}
After the proof of this theorem, on page~\pageref{tab:constant_values} we provide a table with the explicit values of $C_{\mathbf L}^{(k)}$ for $\Z^2$, $\mathbf T$ and $\mathbf H$. The reader will observe that \mbox{$C_{\mathbf H}^{(2)} = 0$}, which means that any cumulant involving $k_v=2$ at any $v$ automatically vanishes on the hexagonal lattice.
\end{remark}

\paragraph{What about neighboring points?}\label{not_good}

A natural question that arises is whether we can relax the good set condition on the set $V$ in theorems~\ref{thm:degree_fgff} and~\ref{thm:cum_discrete}. The answer is yes, as we explain below.

Let $\mathcal G = (\Lambda,\,E)$ be any graph, $T$ a realization of the UST distribution, and $v\sim w\in\Lambda$. Then
\begin{multline*}
    \P\left(D_v = k_v,\, D_w = k_w\right) =\\ \P\left(D_v = k_v,\, D_w = k_w,\, \{v,\,w\}\in T\right) + \P\left(D_v = k_v,\, D_w = k_w,\, \{v,\,w\}\notin T\right) .
\end{multline*}
As we saw in Remark~\ref{rmk:fermions}, the condition $\{v,\,w\}\in T$ translates, in the fermionic language, to introducing the multiplicative factor $\zeta(\{v,\,w\})$, whereas for $\{v,\,w\}\notin T$ we need to introduce $1-\zeta(\{v,\,w\})$. In view of Theorem~\ref{thm:cum_discrete}, we have
\begin{multline*}
    \kappa\left(X_v^{k_v}Y_v,\, X_w^{k_w}Y_w,\, \zeta(\{v,\,w\})\right) = (-1)^{k_v+k_w}
	\sum_{\substack{|\mcE_v|\ge k_v-1\\\{v,\,w\}\notin \mcE_v}} \, \sum_{\substack{|\mcE_w|\ge k_w-1\\\{v,\,w\}\notin \mcE_w}} (-1)^{|\mcE_v|+|\mcE_w|} \times \\ \binom{|\mcE_v|}{k_v-1}\binom{|\mcE_w|}{k_w-1}
	\sum_{\tau\in S_{\co}(\mcE)} \sign(\tau) \prod_{f\in \mcE} M\left(f,\,\tau(f)\right).
\end{multline*}
Equivalently,
\begin{multline*}
    \kappa\left(X_v^{k_v}Y_v,\, X_w^{k_w}Y_w,\, 1-\zeta(\{v,\,w\})\right) = (-1)^{k_v+k_w}
	\sum_{\substack{|\mcE_v|\ge k_v\\\{v,\,w\}\notin \mcE_v}}\sum_{\substack{|\mcE_w|\ge k_w\\\{v,\,w\}\notin \mcE_w}} (-1)^{|\mcE_v|+|\mcE_w|} \times \\ \binom{|\mcE_v|}{k_v}\binom{|\mcE_w|}{k_w}
	\sum_{\tau\in S_{\co}(\mcE)} \sign(\tau) \prod_{f\in \mcE} M\left(f,\,\tau(f)\right).
\end{multline*}
With these expressions we can calculate the moments that give the sought-after probabilities. This is immediately generalized to the case of an arbitrary finite amount of points.

% \todo[inline]{See if I remove this remark}
% \begin{remark}
% One can use this to calculate, for example, the joint probabilities of having degree $k_v$ and $k_w$ at points $v$ and $w$ respectively for the complete graph $K_n$, and take the limit $n\to\infty$ just like we did before. The obtained results do not convey much clarity on its behavior given their convoluted expressions, which cannot be reduced to a product of Poisson probabilities, since a straight-forward calculation shows that $\E[D_v D_w] \neq \E[D_v]\,\E[D_w]$.
% \end{remark}

\paragraph{Complete graphs}\label{complete_graphs}

It is shown in \citet[Theorem 1.3]{pemantle} that, for any complete graph $K_n$ with $n$ vertices, as $n$ goes to infinity the degree of the UST at any vertex $v$ converges in distribution to a random variable $1+\mathcal P(1)$, being $\mathcal P(1)$ a Poisson variable with parameter $1$. This can also be obtained as a corollary from our Theorem~\ref{thm:cum_discrete} in a much shorter way, as follows:

\begin{theorem}[{\citet[Theorem 1.3]{pemantle}}]
    Let $K_n$ be a complete graph with $n$ vertices, and let $V(K_n)$ be its vertex set. For any $v\in V(K_n)$ it holds that
    \[
        D_v \xrightarrow[n\to\infty]{\mathrm{dist}} 1 + \mathcal P(1),
    \]
    with $\mathcal P(1)$ a Poisson random variable with parameter $1$.
\end{theorem}

\begin{proof}[Alternative simpler proof]
From Theorem~\ref{thm:cum_discrete}, for $k = 1,\,\ldots,\,n$ we have that
\[
    \P(D_v = k) = (-1)^k \sum_{\mcE\in E_v:\,\mcE\geq k} (-1)^{|\mcE|} \binom{|\mcE|}{k} \det(M)_\mcE .
\]
According to \citet{pemantle}, the matrix $M$ for a complete graph $K_n$ is given by
\[
    M(e,\,f) = 
    \begin{cases}
        2/n & \text{if } e = f,\\
        1/n & \text{if } e\neq f.            
    \end{cases}
\]
Straightforward calculations then yield
\[
    \det(M)_\mcE = \frac{1+|\mcE|}{n^{|\mcE|}} .
\]
This way,
\[
    \P(D_v = k) = (-1)^k \sum_{\mcE\in E_v:\,\mcE\geq k} (-1)^{|\mcE|} \binom{|\mcE|}{k}  \frac{1+|\mcE|}{n^{|\mcE|}} = (-1)^k \sum_{k'=k}^{n-1} \binom{n-1}{k'} (-1)^{k'} \binom{k'}{k}  \frac{1+k'}{n^{k'}} .
\]
After algebraic manipulations,
\[
    \P(D_v = k) = (1 + k) (n-1)^{-(2+k)} \left(\frac{n-1}{n}\right)^n n \left[n \binom{n-1}{k} - \binom{n}{1 + k}\right].
\]
Taking the limit $n\to\infty$,
\[
    \P(D_v = k) \xrightarrow{n\to\infty} \frac{e^{-1}}{(k-1)!}, \quad k\geq 1,
\]
which exactly matches the distribution of a random variable $1+\mathcal P(1)$.
\end{proof}

\begin{remark}
    As~\citet[Section 5.2]{pemantle} mentions, this result holds for a more general set of graphs, which the author calls \emph{Gino-regular graphs}, and the proof follows in the same way. A sequence of graphs $(\mathcal G_n)_n$ is called Gino-regular if there exists a sequence of positive integers $(D_n)_n$ such that
    \begin{enumerate}[label=(\roman*)]
        \item the maximum and minimum degree of any vertex in $\mathcal G_n$ behaves as $(1+ o(1)) D_n$ as $n\to\infty$, and
        \item the maximum and minimum over vertices $x,\,y,\,z,\,x\neq y$ of $\mathcal G_n$ of the probability that a symmetric random walk on $\mathcal G_n$ started at $x$ hits $y$ before $z$ behaves as $1/2 + o(1)$ as $n\to\infty$,
    \end{enumerate}
    where by $o(1)$ we intend a quantity that vanishes as $n\to\infty$. The set of complete graphs $(K_n)_n$ satisfy these conditions, and so do the $n$-cubes.

    This type of graphs allow for an asymptotic calculation of the determinant of $M$, so that in the limit we obtain the same results as in the case of the complete graph.
\end{remark}

\section{Proofs of Theorems~\ref{thm:degree_fgff}, \ref{thm:cum_discrete} and~\ref{thm:cum_cont}}\label{sec:proofs}

\subsection{Proof of Theorem~\ref{thm:degree_fgff}}

The first equality is trivial from the fact that $\P(X\in A) = \E\left[\1_{\{X\in A\}}\right]$ for any random variable $X$ and any measurable set $A$. Let us then prove the second equality, starting with a simple lemma.

\begin{lemma}\label{lem:degree_kv}
The degree-$k$ fields satisfy
\begin{equation}
\E\left[\prod_{v\in V}\delta^{(k_v)}_v\right] = \sum_{\substack{\eta:\,V\to2^{E_o}\\ |\eta(v)|= k_v\; \forall v\in V}} \P\big(\left\{e \in T \ \forall \, e \in \eta(V)\right\}\, \cap\, \left\{e' \notin T \ \forall \, e'\in E(V) \setminus \eta(V)\right\}\big),
\end{equation}
where $\eta(V)$ is an abuse of notation for $\cup_{v\in V}\eta(v)$.
\end{lemma}

\begin{proof}
This is immediate from the fact that, for any random variable $X$, any $\mathcal I\subset\N$, and measurable sets $A_i$ with $i\in\mathcal I$, $\E\left[\prod_{i\in \mathcal I}\1_{\{X\in A_i\}}\right] = \P\left(\bigcap_{i\in\mathcal I} A_i\right)$.
\end{proof}

Before we proceed with the proof, let us recall Proposition 4.4 from \citet{cipriani2023fermionic}.

\begin{proposition}\label{thm:swan_UST_fGFF}
Let $\mathcal G=(\Lambda,\, E)$ be a finite graph. For all subsets of edges $S\subseteq E$
\begin{equation}\label{eq:UST-grad-ferm}
    \P(T:\, S\subseteq T) = \Ef{\prod_{f\in S}\zeta(f)}.
\end{equation}
\end{proposition}

\begin{proof}[Proof of Theorem~\ref{thm:degree_fgff}]

In view of Lemma~\ref{lem:degree_kv}, take any $\eta:\, V\to 2^{E(V)}$, with $|\eta(v)|= k_v$, $k_v\in\{1,\,\dots,\,\deg_{\mathcal G}(v)\}$, for each $v\in V$. First we observe that
\begin{equation*}
    \bigcap_{v\in V}\left( \left\{\eta(v)\subseteq T \right\}\cap\left(\bigcup_{e\in E_v\setminus\{\eta(v)\}}\{e\in T\}\right)^c\right)=
    \bigcap_{v\in V}\left\{\eta(v)\subseteq T \right\}\cap\left(\bigcup_{e\in E(V)\setminus\{\eta(V)\}}\{e\in T\}\right)^c.
\end{equation*}
By the inclusion--exclusion principle,
% \begin{align}\label{eq:fermio_height}
% &\P \left(\bigcap_{v\in V}\left( \left\{\eta(v)\in T \right\}\cap\left(\bigcup_{e\in E_v\setminus\{\eta(v)\}}\{e\in T\}\right)^c\right)\right) \nonumber \\
% &
% = \P \left( \bigcap_{v\in V}\{\eta(v)\in T \}\right) -  \P \left( \bigcap_{v\in V}\left\{\eta(v) \in T\right\}\cap\bigcup_{e\in E(V)\setminus\{\eta(V)\}}\,\{e\in T\}\right) \nonumber \\
% %&=\sum_{v\in V} \P \left(\eta(v)\in T\right)- \sum_{v\in V}\sum_{\emptyset\neq\alpha\subseteq \widetilde{\mathcal E}(v) \backslash \{\eta(v)\}} (-1)^{|\alpha|-1} \P \big(\left\{\eta(v) \in T\right\}\cap \{\alpha\subseteq T\}\big) \\
% &= \sum_{S \subseteq E(V) \backslash \eta(V)} (-1)^{|S|}\P \left(\bigcap_{v\in V}\{\eta(v) \in T\}\cap (S \subseteq T) \right),
% \end{align}

\begin{multline}
\P \left(\bigcap_{v\in V}\left( \left\{\eta(v)\subseteq T \right\}\cap\left(\bigcup_{e\in E_v\setminus\{\eta(v)\}}\{e\in T\}\right)^c\right)\right) \\
\begin{aligned}
&= \P \left( \bigcap_{v\in V}\{\eta(v)\subseteq T \}\right) -  \P \left( \bigcap_{v\in V}\left\{\eta(v) \subseteq T\right\}\cap\bigcup_{e\in E(V)\setminus\{\eta(V)\}}\,\{e\in T\}\right) \\
%&=\sum_{v\in V} \P \left(\eta(v)\in T\right)- \sum_{v\in V}\sum_{\emptyset\neq\alpha\subseteq \widetilde{\mathcal E}(v) \backslash \{\eta(v)\}} (-1)^{|\alpha|-1} \P \big(\left\{\eta(v) \in T\right\}\cap \{\alpha\subseteq T\}\big) \\
&= \sum_{S \subseteq E(V) \backslash \eta(V)} (-1)^{|S|}\P \left(\bigcap_{v\in V}\{\eta(v) \subseteq T\}\cap (S \subseteq T) \right),\label{eq:fermio_height}
\end{aligned}
\end{multline}
where we sum over the probabilities that the edges of $\eta(V)$ are in the spanning tree $T$ as well as those in $S\subseteq E(V) \backslash \eta(V)$. By Proposition \ref{thm:swan_UST_fGFF}, this becomes

\begin{equation} \label{eq:expansion-XvYv} 
     \sum_{S \subseteq E(V) \backslash \eta(V)} (-1)^{|S|} \Ef{\prod_{\{r,\,s\} \in \eta(V)}\zeta(\{r,\,s\})\prod_{\{u,\,w\}\in S}\zeta(\{u,\,w\})}.
\end{equation}
By the anticommutation relation, the sets of edges $S$ such that $S \cap \eta(V) \neq \emptyset$ do not contribute to~\eqref{eq:expansion-XvYv}. This way,

\begin{multline*}
  \sum_{S \subseteq E(V) } \Ef{\prod_{\{r,\,s\}\in \eta(V)}\zeta(\{r,\,s\})\prod_{\{u,\,w\}\in S}(-1)^{|S|}\zeta(\{u,\,w\})}\\
  \begin{aligned}
  &=\Ef{\prod_{\{r,\,s\}\in \eta(V)}\zeta(\{r,\,s\})\sum_{S\subseteq E(V) }\prod_{\{u,\,w\}\in S}(-1)^{|S|}\zeta(\{u,\,w\})}\\
  &=\Ef{\prod_{\{r,\,s\}\in \eta(V)}\zeta(\{r,\,s\})\prod_{\{u,\,w\}\in E(V) }\big(1-\zeta(\{u,\,w\})\big)}.
  \end{aligned}
\end{multline*}
Observing that the first product is
\[
\prod_{\{r,\,s\}\in \eta(V)}\zeta(\{r,\,s\}) = \prod_{v\in V} \prod_{e\in \eta(v)}\zeta(e)
\]
and summing over all possible such $\eta$'s, we obtain the result.
\end{proof}

\subsection{Permutations, graphs and partitions}\label{subsec:graphs-and-permutations}

In this subsection, we introduce more notation used in Theorem~\ref{thm:cum_discrete} and the proof of Theorem~\ref{thm:cum_cont}.

\paragraph{General definitions.}\label{par:gen_perm} Let $\Lambda$ be a finite and connected (in the usual graph sense) subset of $\mathbf L$ and $V \subseteq \Lambda$ be a good set according to Definition~\ref{def:goodset}. As $V$ is a good set, notice that every edge in $E(V)$ is connected to exactly one vertex in $V$. 

For any finite set $A$ we denote the set of permutations of $A$ by $S(A)$.
Furthermore, we write $S_{\cycl}(A)$ to denote the set of \emph{cyclic} permutations of $A$ (without fixed points). %Finally recall the set $\Pi(A)$ of partitions of $A$.

% We will use the natural partial order between partitions \citep[Chapter 2]{taqqu}; that is, given two partitions $\pi$, $\tilde{\pi}$, we say that $\pi \le \tilde{\pi}$ if for every $B \in \pi$ there exists a $\tilde{B} \in \tilde{\pi}$ such that $B \subseteq \tilde{B}$. If $\pi \le \tilde{\pi}$, we say that $\pi$ is finer than $\tilde{\pi}$ (or that $\tilde{\pi}$ is coarser than $\pi$). If $\sigma \in S(V)$, we denote as $\pi_\sigma \in \Pi(V)$ the partition given by the disjoint cycles of $\sigma$. This is the finest partition such that $\sigma(B) = B$ for all $B \in \pi_\sigma$.
%\\

\paragraph{Permutations: connected and bare.}\label{par:edge_ver}

We define the multigraph $V_\tau= \left(V,\, E_\tau(V)\right)$ induced by $\tau$ in the following way. For each pair of vertices $v \neq w$ in $V$,  we add one edge between $v$ and $w$ for each $f \in E_v, f' \in E_w$ such that either $\tau (f)= f'$ or $\tau (f')= f$. If  $v=w$, we add no edge, so $\deg_{V_\tau}(v) \le |E_v|$. 

Fix $A\subseteq E(V)$ such that $E_v \cap A \neq \emptyset$ for all $v \in V$, i.e. we have a set of edges with at least one edge per vertex of $V$.
Let $\tau \in S(A)$ be a permutation of edges in $A$.

\begin{definition}[Connected and bare permutations]\label{def:connected-permutations}
	Let $\Lambda \subseteq \mathbf L$ finite, $V$ good as in Definition~\ref{def:goodset}, $|V|\geq2$, $A\subseteq E(V)$ and 
	$\tau \in S(A)$ be given.
	\begin{itemize}
		\item We say that $\tau$ is \emph{connected} if the multigraph $V_\tau$ is a connected multigraph.
  
		\item We say that $\tau$ is \emph{bare} if it is connected and $\deg_{V_\tau}(v)=2$ for all $v\in V$ (it is immediate to see that the latter condition can be replaced by $|E_\tau(V)|=|V|$). 
	\end{itemize}
    If $|V|=1$, as it can happen in Theorem~\ref{thm:cum_discrete}, we consider every permutation $\tau\in S(A)$ as both connected and bare.
 
	We will denote by $S_{\co}(A)$ the set of connected permutations in $S(A)$, and by $S_{\bare}(A)$ the set of bare permutations. See Figures~\ref{fig:connected} and~\ref{fig:bare} for some examples, where the mapping $\tau(f)=f'$ is represented via an arrow $f\to f'.$
\end{definition}

\begin{figure}[ht!]
\begin{minipage}{0.45\textwidth}
    \centering
    \includegraphics[scale=.6]{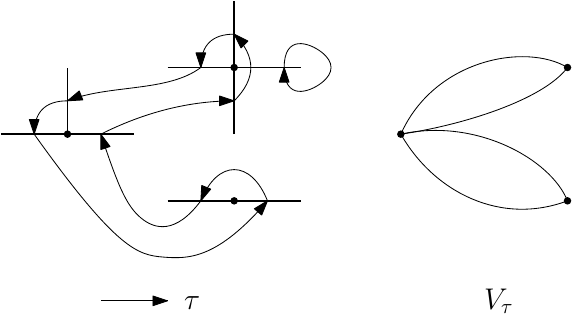}
    \caption{A connected permutation $\tau$ on edges of $\Z^2$ and the multigraph $V_\tau$ associated to it.
    Notice that this permutation is \textbf{not} bare.}
    \label{fig:connected}
\end{minipage}\hfill
\begin{minipage}{0.45\textwidth}
    \centering
    \includegraphics[scale=0.6]{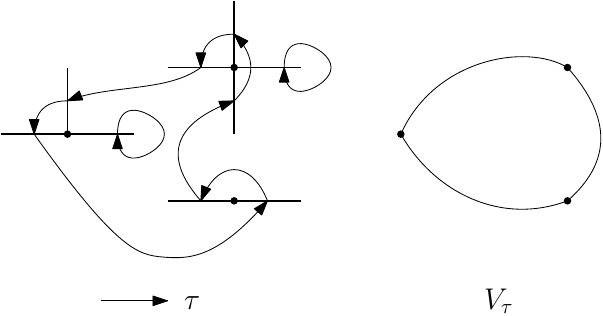}
    \caption{A bare permutation $\tau$ on edges of $\Z^2$ and the multigraph $V_\tau$ associated to it.}
    \label{fig:bare}
\end{minipage}
\end{figure}

For $\tau$ bare we have, by definition, that for each $v$ there are exactly two edges $f,\,f^\prime \in A$ (possibly the same) such that $\tau (f') \not \in E_v$ and $\tau^{-1} (f) \not \in E_v$. We will refer to this as $\tau$ enters $v$ through $f$ and exits $v$ through $f^\prime$. Therefore, for any bare permutation $\tau \in S_{\bare}(A)$, we can define an induced permutation on vertices $\sigma = \sigma_\tau \in S_{\cycl}(V)$ given by $\sigma(v)=w$ if there there exists (a unique) $f \in E_v$ and $f' \in E_w$ such that $\tau(f)=f'$. Figure~\ref{fig:bare_to_sigma} shows an example in $\Z^2$.

\begin{figure}[ht!]
    \centering
    \includegraphics[scale=.6]{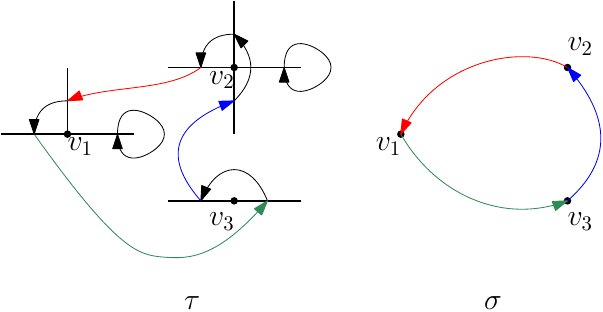}
    \caption{A bare permutation $\tau$ on edges and the induced permutation $\sigma$ on points, in $d=2$.}
    \label{fig:bare_to_sigma}
\end{figure}

% Any permutation of $\tau \in S(A)$ induces a partition $\pi_\tau$ on $A$ given by the disjoint cycles in $\tau$. Likewise, the partition $\pi^V_\tau \in \Pi(V)$  given by the connected components of $V_\tau$ gives the finest partition on $V$ such that $\tau \left(\cup_{v \in B}  A_v\right) = \cup_{v \in B} A_v$ for all $B \in \pi^V_\tau$. If $\tau$ is connected, then $\pi^V_\tau = V$. 

\subsection{Proof of Theorem~\ref{thm:cum_discrete}}

\begin{proof}[Proof of Theorem~\ref{thm:cum_discrete}]
Call $Z_v^{(k_v)}\coloneqq X_v^{k_v}\, Y_v$. Using the same arguments as in the proof of \citet[Theorem 3.5]{cipriani2023fermionic} we get
\begin{multline*}
\kappa\left(Z_{v_1}^{(k_{v_1})},\,\dots,\,Z_{v_n}^{(k_{v_n})}\right) 
	= \\
	\sum_{\eta}
	\sum_{A} (-1)^{|A|}
	\sum_{\pi\in\Pi(V)} \left(|\pi|-1\right)\!!\, (-1)^{|\pi|-1} \prod_{B\in\pi} \sum_{\tau\in S(E_B)} \sign(\tau) \prod_{f\in E_B} M\left(f,\,\tau(f)\right),
\end{multline*}
where the sum over $\eta$'s is over all functions $\eta:\, V \to E(V)$ with $\eta(v)\in E_v $ for all $v$, the sum over $A$'s is over the subsets of $A \subseteq {E}(V) \setminus \eta(V)$, and $E_B=E_B(\eta,\,A)$ is the set of edges in $\eta(V) \cup A$ that intersect sites of $B$.

Notice that $|A|=|\eta(B) \cup A| - \sum_v k_v$. Therefore, the sum above only depends on $\eta$ and $A$ through $\eta(B) \cup A$. We then denote $\mcE=E(\eta,\,A) \coloneqq \eta(V) \cup A$ and recall $\mcE_B = \{f\in \mcE:\, \{f^-\}\cap  B\neq \emptyset\}$. For $v\in V$ we will simplify notation by writing $\mcE_{v}$ rather than $\mcE_{\{v\}}$.

We notice that for a fixed $\mcE$ there are $ \prod_{v \in V} \binom{|\mcE_{v}|}{k_v}$ choices for $\eta(V)$ and $A$ yielding the same $\mcE$, so the sum above can be written as
\begin{multline*} 
\kappa\left(Z_v^{(k_v)}:\, v \in V\right) 
= \\
(-1)^{\sum_v k_v} \sum_{\mcE:\, |\mcE_{v}| \ge k_v\; \forall v} K(\mcE) \sum_{\pi\in\Pi(V)} 
\left(|\pi|-1\right)\!!\, (-1)^{|\pi|-1} \prod_{B\in\pi} \sum_{\tau\in S(\mcE_B)} \sign(\tau) 
\prod_{f\in \mcE_B} M\left(f,\,\tau(f)\right).
\end{multline*}

The sum over partitions $\Pi(V)$ can again be treated in much the same way as \citet[Theorem 3.5]{cipriani2023fermionic}, yielding
\[
	\kappa\left(Z_v^{(k_v)}:\, v \in V\right) 
	= 
	(-1)^{\sum_v k_v}
	\sum_{\mcE:\, |\mcE_{v}|\ge k_v \; \forall v} K(\mcE) 
	\sum_{\tau\in S_{\co}(\mcE)} \sign(\tau) \prod_{f\in \mcE} M\left(f,\,\tau(f)\right)
\]
as we wanted to show.
\end{proof}

\subsection{Proof of Theorem~\ref{thm:cum_cont}}

\begin{proof}[Proof of Theorem~\ref{thm:cum_cont}]

We will do a general proof that works for both $\mathbf L = \Z^d$ and $\mathbf L = \mathbf T$ (and $\mathbf H$ with an exception that we will mention below). The proof is divided into four steps.
In \ref{step1}, we start from the final expression obtained in Theorem~\ref{thm:cum_discrete} and show that it suffices to sum over only bare permutations $\tau$, instead of the bigger set of connected permutations.
% In~\ref{step2}, we simplify the expression even further, showing that only permutations that enter and exit through parallel edges on every point will give a non-zero contribution in the scaling limit as $\eps\rightarrow 0$.
In~\ref{step2}, we write the expression in terms of contributions of the permutations acting locally in the vicinity of a vertex and globally mapping an edge incident to one vertex to an edge which is incident to another vertex.
In~\ref{step3} we argue that, given a permutation $\tau$ on edges and an entry edge for any given point $v\in V$, only the projection of the exit edge onto the entry edge will contribute to the final expression, so we can treat the former as a new edge in the direction of the entry one, weighed by its projection.
Finally, in~\ref{step4}, we identify the global multiplicative constant of the cumulants.

\begin{enumerate}[wide, labelindent=0pt, label={\bf Step \arabic*.}, ref=Step \arabic*]

\item\label{step1}
% \todo[inline]{This step is identical to the previous paper. Make it more concise and refer to previous paper. The part of the last display that depends on $k_v$ doesn't matter here, and we only care about the last sum and prod, which is identical to previous paper.}
From Theorem~\ref{thm:cum_discrete} we start with the expression
\[
	\kappa\left(\left(Z_v^{(k_v)}\right)^\eps:\, v \in V\right) 
	= 
	(-1)^{\sum_v k_v}
	\sum_{\mcE:\, |\mcE_{v_\eps}|\ge k_v \; \forall v} K(\mcE) 
	\sum_{\tau\in S_{\co}(\mcE)} \sign(\tau) \prod_{f\in \mcE} M\left(f,\,\tau(f)\right).
\]
This step is practically identical to Step 1 in the proof of Theorem 3.6 in~\citet{cipriani2023fermionic}, since it does not depend on $k_v$, so we omit the whole derivation. It is obtained that, in the limit $\eps\to0$, only \emph{bare} permutations contribute to the final result, obtaining the expression
\begin{equation}\label{eq:now_bare}
	(-1)^{\sum_v k_v}
	\sum_{\mcE:\, |\mcE_{v}|\ge k_v \; \forall v}  K(\mcE) 
	\sum_{\tau\in S_{\bare}(\mcE)} \sign(\tau) \prod_{f\in \mcE} \overline M\left(f,\,\tau(f)\right),
\end{equation}
where we use the notation
\begin{equation}\label{eq:def_bar_M}
\overline{M}(f,\,\tau(f)) = 
\begin{dcases}
    \hfil \nabla^{(1)}_{e_i}\nabla^{(2)}_{e_j}G_0(o,\,o)&\text{if } f^- = \tau(f)^-,\\ 
    \partial^{(1)}_{e_i}\partial^{(2)}_{e_j} g_U\left(v,\, v'\right)  &\text{if } f^-=v_\eps \neq v'_\eps=\tau(f)^-,\ v,\,v'\in V
\end{dcases}
\end{equation}
whenever $\eta^*(f)=e_i$ and $\eta^*(\tau(f)) = e_j$ for some $e_i,\,e_j\in E_o$.

\begin{remark}
In the hexagonal lattice there are two types of points: those with edges at $0$, $2\pi/3$ and $4\pi/3$ degrees, and those with edges at $\pi/3$, $\pi$ and $5\pi/6$ degrees. Following the proof in~\citet{cipriani2023fermionic}, this step needs extra care when dealing with the hexagonal lattice, since as $\eps\to0$ $v_\eps$ alternates between the two different types of points. Nevertheless, regardless of the point, the contribution will be the same and the result holds for $\mathbf H$ as well, but we omit this technical detail.
\end{remark}

\item\label{step2}

% Now, we examine the limiting expression
% %
% \begin{equation}\label{eq:now_bare}
% 	(-1)^{\sum_v k_v}
% 	\sum_{\mcE:\, |\mcE_{v}|\ge k_v \; \forall v}  K(\mcE) 
% 	\sum_{\tau\in S_{\bare}(\mcE)} \sign(\tau) \prod_{f\in \mcE} \overline M\left(f,\,\tau(f)\right).
% \end{equation}
%
% For the remainder of the proof, we will use the notation
% \begin{equation}\label{eq:def_bar_M}
% \overline{M}(f,\,\tau(f)) = 
% \begin{dcases}
%     \hfil \nabla^{(1)}_{e_i}\nabla^{(2)}_{e_j}G_0(o,\,o)&\text{if } f^- = \tau(f)^-,\\ 
%     \partial^{(1)}_{e_i}\partial^{(2)}_{e_j} g_U\left(v,\, v'\right)  &\text{if } f^-=v_\eps \neq v'_\eps=\tau(f)^-,\ v,\,v'\in V
% \end{dcases}
% \end{equation}
% whenever $\eta^*(f)=e_i$ and $\eta^*(\tau(f)) = e_j$ for some $e_i,\,e_j\in E_o$.

Given $\tau \in S_{\bare}(\mcE)$, fix $v\in V$, and let $\eta(v)=\eta(v,\,\tau)$ be the edge through which $\tau$ enters $v$. Let $\alpha(v) \in \{0,\,\dots,\,p - 1\}$, where $p$ is the number of edges contained in any two dimensional plane generated by any two edges incident on any $v\in V$; that is, $4$ for the hypercubic lattice in $d$ dimensions and $6$ for the triangular lattice. We define $\eta^\alpha(v)$ as the edge through which $\tau$ exists $v$, and $2\pi\alpha(v)/p$ denotes the angle between the entry and exit edges. Let $\gamma_\alpha(v) \coloneqq \cos \left(2\pi\alpha(v)/p\right)$, so that
\[
    \langle\eta(v),\,\eta^\alpha(v)\rangle = \gamma_\alpha(v).
\]
In the case of the hypercubic lattice the angles between entry and exit edges are multiples of $\pi/2$, hence their cosines belong to $\{-1,\,0,\,1\}$, whereas in the triangular lattice in $d=2$ angles are multiples of $\pi/3$, and their cosines belong to $\{-1,\,-1/2,\,0,\,1/2,\,1\}$.

As stated in Subsection~\ref{subsec:graphs-and-permutations}, any bare $\tau$ induces a permutation $\sigma\in S_{\cycl}(V)$ on vertices. We will extract from $\tau$ a permutation $\sigma$ among vertices and a choice of edges $\eta$, and we will separate it from what $\tau$ does ``locally'' in the edges corresponding to a given point. Note that $\eta$, $\sigma$ and $\alpha$ determine $E_\tau(V)$ and are functions of $\tau$ (we will not write this to avoid heavy notation). With the above definitions we have that~\eqref{eq:now_bare} becomes
\begin{multline}
(-1)^{\sum_v k_v}
	\sum_{\mcE:\, |\mcE_{v}|\ge k_v \; \forall v} 
	\sum_{\substack{\eta:\,V\to E(V)\\ \eta(v)\in \mcE_v\; \forall v}}\;
	\sum_{\sigma\in S_{\cycl}(V)} \sum_{\alpha:\,V\to\{0,\,\ldots,\,p-1\}} 
	\left(\prod_{v\in V} K(\mcE_v)\overline{M}\left(\eta^\alpha(v),\,\eta(\sigma(v))\right)\right)\times 
	\\
	\times\sum_{\tau\in S_{\bare}(\mcE;\,\eta,\,\sigma,\,\alpha)}\sign(\tau) \prod_{f\in \mcE\setminus \{\eta^\alpha(V)\}} 
	\overline{M}\left(f,\,\tau(f)\right),\label{eq:first_lim}
\end{multline}
where $\eta^\alpha(V) \coloneqq \{\eta^\alpha(v):\,v\in V\}$, and $S_{\bare}(\mcE;\,\eta,\,\sigma,\,\alpha)$ is the set of bare permutations which now enter and exit each point $v$ through the edges prescribed by $\eta$, $\sigma$ and $\alpha$. In this case we will say that $\tau$ is compatible with $(\mcE;\,\eta,\,\sigma,\,\alpha)$. Figures~\ref{fig:eta_sigma_gamma_2_figures_compatible} and \ref{fig:eta_sigma_gamma_incompatible_4_figures} give examples of compatible resp. non-compatible pairs of permutations for the hypercubic lattice in $d=2$.

\begin{figure}[ht!]
    \centering
    \includegraphics[scale=.6]{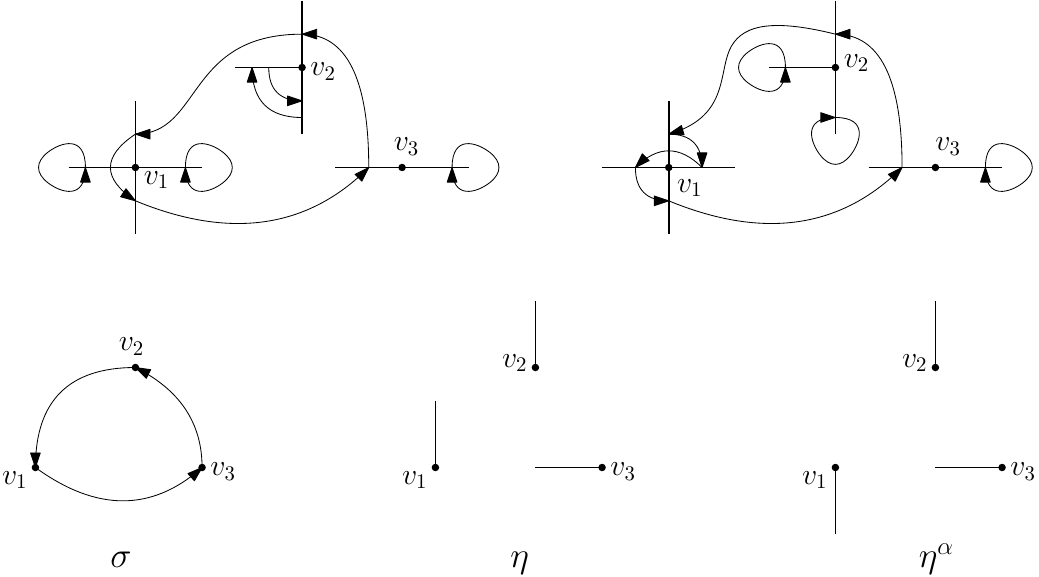}
    \caption{Top: two different compatible permutations in the hypercubic case in $d=2$. Bottom: their corresponding $\sigma$, $\eta$ and $\eta^\alpha$.}
    \label{fig:eta_sigma_gamma_2_figures_compatible}
\end{figure}

\begin{figure}[ht!]
    \centering
    \includegraphics[scale=.6]{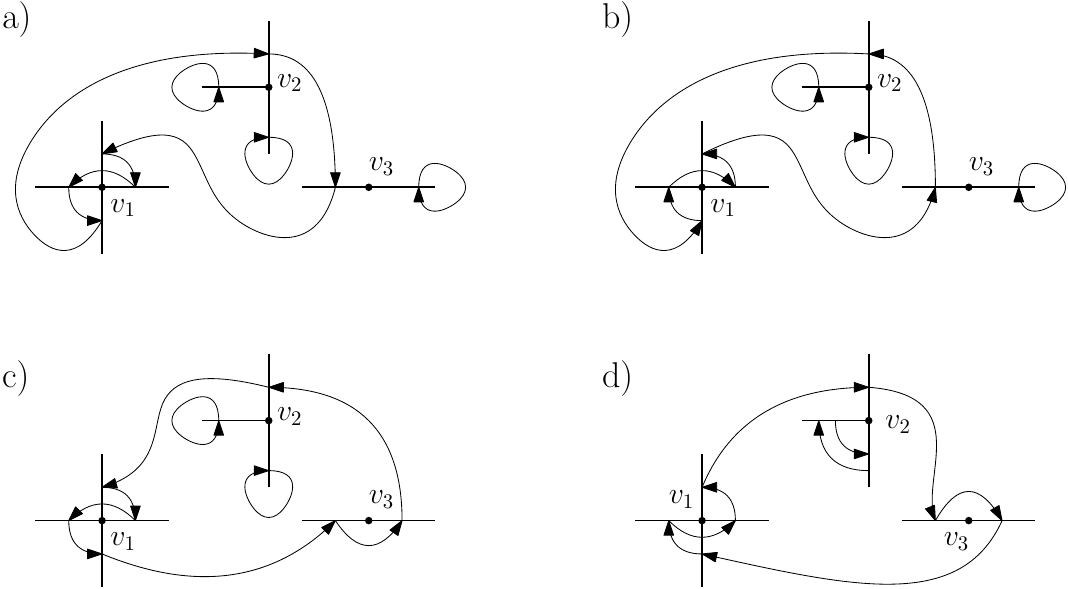}
    \caption{Four different permutations that are not compatible with those in Figure~\ref{fig:eta_sigma_gamma_2_figures_compatible}. a) Permutation that respects $\eta$ and $\eta^\alpha$ but not $\sigma$. b) Permutation that respects $\sigma$ and $\eta^\alpha$ but not $\eta(v_1)$. c) Permutation that respects $\sigma$ and $\eta$ but not $\eta^\alpha(w_3)$. d) Permutation that does not respect $\sigma$, nor $\eta(v_1)$, nor $\eta^\alpha(w_3)$.}
    \label{fig:eta_sigma_gamma_incompatible_4_figures}
\end{figure}

\item\label{step3}

Define $R_{v,\,\eta,\,\alpha}:\, \mathbb{R}^d \to \mathbb{R}^d$ to be the reflection perpendicular to the line given by $\eta(v)$, parallel to the plane generated by $\eta(v)$ and $\eta^\alpha(v)$ (in case they are co-linear the reflection is the identity). More precisely, let us call $\mathfrak S$ the plane generated by $\eta(v)$ and $\eta^\alpha(v)$, assuming they are not co-linear. Any edge $e\in\mcE$ can always be decomposed as
\[
    e = \mathcal P^{\mathfrak S}(e) + \mathcal P^{\mathfrak S^\perp}(e),
\]
being $\mathcal P^{\mathfrak S}$ (resp.~$\mathcal P^{\mathfrak S^\perp}$) the orthogonal projection operator on $\mathfrak S$ (resp. $\mathfrak S^\perp$, that is, the orthogonal complement of $\mathfrak S$ on $\R^d$). In turn, this can be further decomposed as
\[
    e = \mathcal P^{\mathfrak S}(e)_{\eta(v)} + \mathcal P^{\mathfrak S}(e)_{\eta(v)^\perp} + \mathcal P^{\mathfrak S^\perp}(e),
\]
being $\mathcal P^{\mathfrak S}(e)_{\eta(v)}$ the component of $\mathcal P^{\mathfrak S}(e)$ in the direction of $\eta(v)$, and $\mathcal P^{\mathfrak S}(e)_{\eta(v)^\perp}$ its orthogonal complement. Of course, $\mathcal P^{\mathfrak S}(e)_{\eta(v)} =(e)_{\eta(v)}$, that is, the component (or projection) of $e$ in the direction of $\eta(v)$. Let us rewrite this as
\[
    e = \mathcal P^{\mathfrak S}(e)_{\eta(v)^\perp} + e'
\]
for some unique $e'\in\R^d$. We then define $R_{v,\,\eta,\,\alpha}:\, \mathbb{R}^d \to \mathbb{R}^d$ as
\[
    R_{v,\,\eta,\,\alpha}(e) \coloneqq -\mathcal P^{\mathfrak S}(e)_{\eta(v)^\perp} + e'.
\]
We then define
\[
    \mathcal{E}' \coloneqq R_{v,\,\eta,\,\alpha}(\mcE)\coloneqq \left(\bigcup_{v' \neq v} \mcE_{v'}\right) \cup \left\{R_{v,\,\eta,\,\alpha}(e):\, e \in \mcE_v\right\}
\]
and, for $\tau \in S_{\bare}(\mathcal{E})$, define $\rho \in S_{\bare}(\mathcal{E}')$ as
\begin{equation*}
	\rho(e) 	
	=
	\begin{cases}
	\tau(e) & \text{ if } e \in \cup_{v' \neq v}\mathcal{E}_{v'},\\
	\tau(\eta^\alpha (v)) &\text{ if } e = R_{v,\,\eta,\,\alpha}(\eta^\alpha (v)),\\
	R_{v,\,\eta,\,\alpha}(\tau(e')) &\text{ if } e=R_{v,\,\eta,\,\alpha}(e') \text{ for some } e'\in\mcE_v\setminus\{\eta^\alpha(v)\} .
	\end{cases} 
\end{equation*}
See Figure~\ref{fig:reflection_square} for an example of the reflected permutation $\rho$ in the square lattice, and Figure~\ref{fig:reflection_triang} for the triangular lattice.
\begin{figure}[ht!]
\begin{minipage}{0.45\textwidth}
    \centering
    \includegraphics[scale=.8]{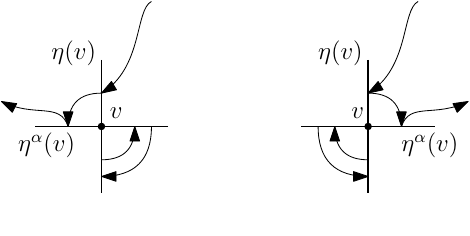}
    \caption{Square lattice in $d=2$. Left: a permutation $\tau$ on $v$. Right: its reflection $\rho$.}
    \label{fig:reflection_square}
\end{minipage}\hfill
\begin{minipage}{0.45\textwidth}
    \centering
    \includegraphics[scale=0.8]{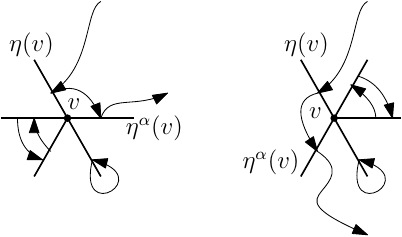}
    \caption{Triangular lattice in $d=2$. Left: a permutation $\tau$ on $v$. Right: its reflection $\rho$.}
    \label{fig:reflection_triang}
\end{minipage}
\end{figure}
We can then see that $K(\mathcal{E})=K(\mathcal{E}')$ and $\sign(\tau)=\sign(\rho)$.
Furthermore, with simple calculations of inner products we have
\begin{equation}\label{eq:cancel-trig-1} 
\overline{M}\left(\eta^\alpha(v),\,\eta(\sigma(v))\right)
+
\overline{M}\left(R_{v,\eta}(\eta^\alpha(v)),\,\eta(\sigma(v))\right)
=
2\cos\left(\frac{2\pi\alpha(v)}{p}\right)\,\overline{M}(\eta(v),\,\eta(\sigma(v))).
\end{equation}
Observe that these cancellations happen in the hypercubic, triangular and hexagonal lattices due to their high symmetries.

With~\ref{eq:cancel-trig-1} in mind, Equation~\eqref{eq:first_lim} becomes
\begin{multline}\label{eq:cancel-trig-2} 
	(-1)^{\sum_v k_v}
	\sum_{\mcE :\, |\mcE_v|\ge k_v \, \forall v} \sum_{\substack{\eta:\,V\to E(V)\\ \eta(v)\in \mcE_v\; \forall v}}\;\sum_{\sigma\in S_{\cycl}(V)} \sum_{\alpha:\,V\rightarrow \{0,\,\ldots,\,p-1\}} 
	\sum_{\tau\in S_{\bare}(\mcE;\,\eta,\,\sigma,\,\alpha)} \sign(\tau) \times
\\ \times 
\prod_{f\in \mcE \setminus \eta^{\alpha}(V)} \overline{M}\left(f,\,\tau(f)\right) \prod_{v\in V}K(\mcE_v) \gamma_\alpha(v) 	\underbrace{\prod_{v\in V} 
	\partial^{(1)}_{\eta(v)}\partial^{(2)}_{\eta(\sigma(v))}
	g_{U}\left(\eta (v),\,\eta(\sigma(v))\right)}_{(\star)}.
\end{multline}
The factor $(\star)$, which accounts for the interactions between different points, only depends on the entry directions given by $\eta$, not on the exit directions $\eta^{\alpha}$.
This is the key cancellation to obtain expressions of the form \eqref{eq:cum_limit3}, up to constant.

We rewrite expression \eqref{eq:cancel-trig-2} as
\begin{multline}\label{eq:second_lim}
\sum_{\substack{\eta:\,V\to E(V)\\ \eta(v)\in E_{v}\; \forall v}}\;
\sum_{\sigma\in S_{\cycl}(V)} 
\prod_{v\in V} 
\partial^{(1)}_{\eta(v)}\partial^{(2)}_{\eta(\sigma(v))}
	g_{U}\left(\eta (v),\,\eta(\sigma(v))\right)
\times \\ \times
\prod_{v\in V}\underbrace{(-1)^{k_v}
	\sum_{\substack{\mcE_v:\, \mcE_v\ni \eta(v)\\|\mcE_v|\geq k_v}} K(\mcE_v)
	\sum_{\alpha=0}^{p-1} \gamma_\alpha(v)
			 \sum_{\tau\in S_{\bare}(\mcE;\,\eta,\,\sigma,\,\alpha)}\sign(\tau) 
			\prod_{f\in\mcE_v\setminus\{\eta^\alpha(v)\}}\overline{M}\left(f,\,\tau(f)\right)}_{(\star\star)}.
\end{multline}
Remark that if $\eta^\alpha(v) \not \in {\mathcal{E}}$, the set  $S_{\bare}(\mcE;\,\eta,\,\sigma,\,\alpha)$ is empty, and therefore not contributing to the sum.

Notice that all entries of the type $\overline{M}(e,\,\tau(e))$ in $(\star \star)$ are discrete double gradients of the Green function of the full lattice $\mathbf{L}$ (see Equation \eqref{eq:def_bar_M}).
In the following we will prove that $(\star \star)$ does not depend on the choice of $\eta$ nor $\sigma$.
The value of the term $(\star \star)$ will give the constants $C_\mathbf{L}^{(k_v)}$ (up to an overall minus sign).

\item\label{step4} 
Using $\sigma$, $\eta$ and $\alpha$, we have been able to isolate in~\eqref{eq:second_lim} an expression that depends only on permutations of vertices. To complete the proof we will perform a ``surgery'' to better understand expression~\eqref{eq:second_lim}. This surgery aims at decoupling the local behavior of $\tau$ at a vertex versus the jumps of $\tau$ between different vertices.

To do this, given $\eta:\, V \to E(V)$, $\alpha: V \to \{0,\,\dots,\,p-1\}$, $\mathcal{E} \subseteq E(V)$ with $\eta(v),\, \eta^\alpha(v)  \in \mathcal{E}_v$, and $\tau \in S_{\bare}(\mcE;\,\eta,\,\sigma,\,\alpha)$, we define $\omega_v^\tau (\mathcal{E}_v \setminus \{\eta(v)\})$ and $\tau \setminus \omega^\tau_v ( (\mathcal{E} \setminus \mathcal{E}_v )\cup \{\eta(v)\})$ as
\begin{equation}\label{eq:def_omega_trig}
\omega^\tau_v(f)\coloneqq
\begin{cases}
    \tau(f)&\text{if }f\neq \eta^\alpha(v)\\
    \tau(\eta(v))&\text{if }f= \eta^{\alpha}(v),\,\alpha(v) \neq 0
\end{cases},\quad f\in \mcE_v\setminus\{\eta(v)\}
\end{equation}
and
\begin{equation*}\label{eq:def:tau_minus_trig}
\tau\setminus\omega^\tau_v(f)\coloneqq \begin{cases}
    \tau(f)&\text{if }f\notin \mcE_v\\
    \eta(\sigma(v))&\text{if }f= \eta(v)
\end{cases},\quad f\in (\mcE\setminus \mcE_v) \cup \{\eta(v)\}.
\end{equation*}
In words, $\omega^\tau_v$ is the permutation induced by $\tau$ on $\mcE_v\setminus\{\eta(v)\}$ by identifying the entry and the exit edges. On the other hand, $\tau\setminus\omega^\tau_v(f)$ follows $\tau$ globally until it reaches the edges incident to $v_\eps$, from where it departs reaching the edges of the next point. An example of $\omega_v^{\tau}$ for the triangular lattice can be found in Figure~\ref{fig:surgery_triangular}.

\begin{figure}[htb!]
    \centering
    \includegraphics[scale=.8]{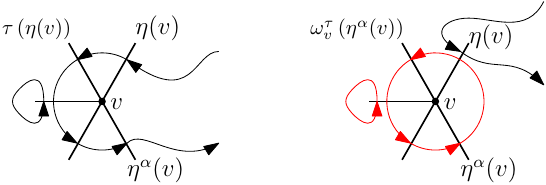}
    \caption{Left: a permutation $\tau$ at the point $v$. Right: the surgery applied to $\tau$, with $\omega^\tau_v$ denoted in red. }
    \label{fig:surgery_triangular}
\end{figure}

In the following we state two technical lemmas the we need to complete the proof of the theorem. These are identical to Lemmas 4.8 and 4.9/5.3 in~\citet{cipriani2023fermionic}, so we omit their proofs.

\begin{lemma}\label{lem:bij_tau}
Let $\mcE\subseteq E(V)$, $ \eta:\,V\to E(V)$ such that $\eta(v)\in \mcE_v$ for all $v\in V$, $\sigma\in S_{\cycl}(V)$, $\alpha:\, V\to\{0,\,\ldots,\,p-1\}$ and let $\tau$ be compatible with $(\mcE;\,\eta,\,\sigma,\,\alpha)$. For every $v\in V$ there is a bijection between $S(\mcE_v\!\setminus\!\{\eta(v)\})$ and $\{\omega^\tau_v:\,\tau\text{\normalfont{ compatible with }} (\mcE;\,\eta,\,\sigma,\,\alpha)\}$.
\end{lemma}

\begin{lemma}[Surgery of $\tau$]\label{lem:surg_tau_trig}
Fix $v\in V$ and $\mcE$, $\eta$, $\sigma$, $\alpha$ as above.
Let $\tau$ be compatible with $\mcE$, $\eta$, $\sigma$ and $\alpha$.
Then 
\begin{equation}\label{eq:sgn_tau}
   \sign(\tau)=(-1)^{\1_{\{\alpha(v) \neq 0\}}}\sign(\tau\setminus\omega^\tau_v(f))\sign(\omega_v^\tau).
\end{equation}
Furthermore,
\[
\prod_{f\in \mcE_{v}\setminus  \{\eta^\alpha(v)\}} \overline M\left(f,\,\tau(f)\right) = \frac{\overline{M}\left(\eta(v),\,\omega^\tau_v(\eta^\alpha(v))\right)}{\overline{M}\left(\eta^\alpha(v),\,\omega^\tau_v(\eta^\alpha(v))\right)}  \prod_{f\in \mcE_{v}\setminus \{\eta(v)\}} \overline M\left(f,\,\tau(f)\right).
\]
Equivalently, we can write that
\begin{equation}\label{eq:surg_M}
\prod_{f\in \mcE_v\setminus \{\eta^\alpha(v)\}} \overline M\left(f,\,\tau(f)\right) = \prod_{f\in \mcE_v\setminus  \{\eta(v)\}} \overline M^{\alpha}\left(f,\,\omega^\tau_v(f)\right),
\end{equation}
where for any $g\in \mcE_v$
\begin{equation}\label{eq:defmalp}
\overline M^{\alpha}\left(f,\,g\right) :=\begin{cases}
\overline{M}(\eta(v),\,g) & \text{ if } f=\eta^{\alpha}(v), \\
\overline{M}(f,\,g) & \text{ if } f\neq \eta^{\alpha}(v).
\end{cases}
\end{equation}
\end{lemma}

Remark that the matrix $\overline{M}^{\alpha}$ is not symmetric anymore. We will now use these lemmas to rewrite~\eqref{eq:second_lim} in a more compact form. Using~\eqref{eq:sgn_tau} recursively, we get
\[
\sign(\tau) = \left(\prod_{v\in V} (-1)^{\1_{\{\alpha(v)\neq 0\}}} \sign(\omega^\tau_v)\right) \sign((((\tau\setminus\omega^\tau_{v_1})\setminus\omega^\tau_{v_2})\setminus\ldots)\setminus\omega^\tau_{v_n}).
\]
Note that the permutation $ (((\tau\setminus\omega^\tau_{v_1})\setminus\omega^\tau_{v_2})\setminus\ldots)\setminus\omega^\tau_{v_n}$ equals the permutation 
\[
(\eta(v_1),\,\eta(\sigma(v_1)),\,\eta(\sigma(\sigma(v_1))),\,\ldots,\,\eta(\sigma^{n-1}(v_1)))
\]
and, as such, it constitutes a cyclic permutation on $n$ edges in $\mcE$, so that
\[
\sign((((\tau\setminus\omega^\tau_{v_1})\setminus\omega^\tau_{v_2})\setminus\ldots)\setminus\omega^\tau_{v_n}) = (-1)^{n-1}.
\]
With this in mind, applying~\eqref{eq:surg_M} at every $v$ we can rewrite $\prod_{v\in V}(\star\star)$ as
\begin{multline*}%\label{eq:make_const}
(-1)^{n-1}\prod_{v\in V}(-1)^{k_v}
	\sum_{\substack{\mcE_v:\, \mcE_v\ni \eta(v)\\|\mcE_v|\geq k_v}} K(\mcE_v)
	\sum_{\alpha=0}^{p-1} \gamma_\alpha(v) \, \1_{\{\eta^\alpha(v)\in\mcE_v\}}\\
			 \sum_{\tau\in S_{\bare}(\mcE;\,\eta,\,\sigma,\,\alpha)}(-1)^{\1_{\{\alpha(v)\neq 0\}}}\sign(\omega_v^\tau)
			\prod_{f\in\mcE_v\setminus\{\eta(v)\}}\overline{M}^\alpha\left(f,\,\omega^\tau_v(f)\right).
\end{multline*}
Recall that, given $\alpha(v)$, $\omega_v^\tau(\eta^\alpha(v))=\tau(\eta(v))$, which means that now the dependence on $\tau$ is only through $\omega_v^\tau$ and $\alpha(v)$. This, together with Lemma~\ref{lem:bij_tau}, allows us to obtain
\begin{multline}\label{eq:prod_const}
-\prod_{v\in V}(-1)^{1+k_v}
	\sum_{\substack{\mcE_v:\, \mcE_v\ni \eta(v)\\|\mcE_v|\geq k_v}} K(\mcE_v)
	\sum_{\alpha=0}^{p-1} \gamma_\alpha(v) \, \1_{\{\eta^\alpha(v)\in\mcE_v\}}\\
			 \sum_{\omega_v\in S(\mcE_v\setminus\{\eta(v)\})}(-1)^{\1_{\{\alpha(v)\neq 0\}}}\sign(\omega_v)
			\prod_{f\in\mcE_v\setminus\{\eta(v)\}}\overline{M}^\alpha\left(f,\,\omega_v(f)\right).
\end{multline}

At this point, we note that the expression above does not depend on $\sigma$ or $\eta$ anymore, and only depends on $v$ through $k_v$.
In fact, as $\omega_v(f)^-=f^{-}=v$, we have that $\overline{M}\left(f,\,\omega_v(f)\right)$ is a constant by definition (see~\eqref{eq:def_bar_M}). 
Therefore, without loss of generality, we can take $v=o$, $\eta(v)=e_1$ to get that \eqref{eq:prod_const} is equal to minus the product over $v$ of
\begin{multline*}%\label{eq:make_const}
(-1)^{1+k_v}
	\sum_{\substack{\mcE_o:\, \mcE_o\ni e_1\\|\mcE_o|\geq k_v}} K(\mcE_o)
	\sum_{\alpha=0}^{p-1}\Bigg[\1_{\{\alpha=0\}}\sum_{\omega\in S(\mcE_o\setminus\{e_1\})}\sign(\omega)
			\prod_{f\in\mcE_o\setminus\{e_1\}}\overline{M}\left(f,\,\omega(f)\right)
 \\
			 - \gamma_\alpha(v) \, \1_{\{e_{1+\alpha}\in\mcE_o\}} \1_{\{\alpha\neq 0\}} \sum_{\omega\in S(\mcE_o\setminus\{e_1\})}\sign(\omega)
			\prod_{f\in\mcE_o\setminus\{e_1\}}\overline{M}^\alpha\left(f,\,\omega(f)\right) \Bigg].
\end{multline*}
Using the definition of determinant of a matrix, after applying the sum on $\alpha\in\{0,\,\ldots,\,p-1\}$ the first term in the square brackets above is equal to $\det\left(\overline M\right)_{\mcE_o\setminus\{e_1\}}$, while for $\alpha \neq 0$ the second one yields $\1_{\{e_{1+\alpha}\in\mcE_o\}} \det\big(\overline M^\alpha\big)_{\mcE_o\setminus\{e_1\}}$, with $\overline M^\alpha$ as in~\eqref{eq:Malpha}. Summing these contributions we obtain the cumulants
\begin{multline*}
- \left[\prod_{v\in V}C_\mathbf{L}^{(k_v)}\right] \left(\frac{1}{c_{\mathbf L}}\right)^n\sum_{\sigma\in S_{\cycl}(V)} \sum_{\eta:\,V\to E_o} \prod_{v\in V} \partial_{\eta(v)}^{(1)}\partial_{\eta(\sigma(v))}^{(2)} g_U\left(v,\, \sigma(v)\right) = \\
- \left[\prod_{v\in V}C_\mathbf{L}^{(k_v)}\right] \sum_{\sigma\in S_{\cycl}(V)} \sum_{\eta:\,V\to \{\tilde e_1,\,\ldots,\,\tilde e_d\}} \prod_{v\in V} \partial_{\eta(v)}^{(1)}\partial_{\eta(\sigma(v))}^{(2)} g_U\left(v,\, \sigma(v)\right),
\end{multline*}
where the last change of coordinates is identical to that of~\citet{cipriani2023fermionic}, being
\[
    C_\mathbf{L}^{(k_v)} = (-1)^{k_v+1} \,c_{\mathbf L}\sum_{\substack{\mcE_o\ni e_1\\|\mcE_o|\geq k_v}} (-1)^{|\mcE_o|} \binom{|\mcE_o|}{k_v} \left[\det\left(\overline M\right)_{\mcE_o\setminus{\{e_1\}}} - \sum_{\alpha=1}^{p-1} \gamma_\alpha \1_{\{e_{1+\alpha}\in \mcE_o\}} \det\big(\overline{M}^\alpha\big)_{\mcE_o\setminus\{e_1\}} \right],
\]
with $c_{\Z^d} = 2$ for all $d \geq 2$, and $c_{\mathbf T} = 3$.\qedhere
\end{enumerate}

\end{proof}

\begin{remark}
    We highlight once again that, with the technical exception of~\ref{step1}, all the other steps follow in much the same way for $\mathbf H$, in which case $p=3$, and the value of $c_{\mathbf H}$ can also be calculated, obtaining $c_{\mathbf H} = 3/2$.
\end{remark}

\label{tab:constant_values}
Using the potential kernel values of the lattices (see e.g. \cite{KenyonWilson} or \cite{RuelleTriang}), some values of $C_{\mathbf L}^{(k_v)}$ in two dimensions are
\begin{align*}
    &C_{\Z^2}^{(1)} = \frac{8}{\pi} - \frac{16}{\pi^2} \approx 0.9253 \\
    &C_{\Z^2}^{(2)} = 18 - \frac{72}{\pi} + \frac{96}{\pi^2} \approx 4.8085\\
    &C_{\Z^2}^{(3)} = 2 + \frac{16}{\pi} \approx 7.0930 \\
    &C_{\Z^2}^{(4)} = -2\\
    &C_{\mathbf T}^{(1)} = -\frac{25}{6} - \frac{5\sqrt{3}}{2\pi} + 
    \frac{297}{\pi^2} - \frac{594\sqrt{3}}{\pi^3} + \frac{972}{\pi^4} \approx 1.3443\\
    &C_{\mathbf T}^{(2)} = -\frac{35}{8} + \frac{611\sqrt{3}}{4\pi} - \frac{4077}{2\pi^2} + \frac{3159\sqrt{3}}{\pi^3} - \frac{4860}{\pi^4} \approx -0.1296\\
    &C_{\mathbf T}^{(3)} = \frac{239}{4} - \frac{537\sqrt{3}}{\pi} + \frac{5031}{\pi^2} - \frac{6696\sqrt{3}}{\pi^3} + \frac{9720}{\pi^4} \approx -0.8286\\
    &C_{\mathbf T}^{(4)} = -\frac{599}{6} + \frac{1433\sqrt{3}}{2\pi} - \frac{5832}{\pi^2} + \frac{7074\sqrt{3}}{\pi^3} - \frac{9720}{\pi^4} \approx -0.3339\\
    &C_{\mathbf T}^{(5)} = \frac{247}{4} - \frac{841\sqrt{3}}{2\pi} + \frac{3240}{\pi^2} - \frac{3726\sqrt{3}}{\pi^3} + \frac{4860}{\pi^4} \approx -0.0497\\
    &C_{\mathbf T}^{(6)} = -\frac{105}{8} + \frac{363\sqrt{3}}{4\pi} - \frac{1395}{2\pi^2} + \frac{783\sqrt{3}}{\pi^3} - \frac{972}{\pi^4} \approx -0.0026\\
    &C_{\mathbf H}^{(1)} = \frac34\\
    &C_{\mathbf H}^{(2)} = 0\\
    &C_{\mathbf H}^{(3)} = -\frac34
\end{align*}

% \begin{align*}
%     % &
%     % \begin{aligned}
%     &C_{\Z^2}^{(1)} = \frac{8}{\pi} - \frac{16}{\pi^2} \approx 0.9253 \hspace{1cm}
%     C_{\Z^2}^{(2)} = 18 - \frac{72}{\pi} + \frac{96}{\pi^2} \approx 4.8085\\
%     &C_{\Z^2}^{(3)} = 2 + \frac{16}{\pi} \approx 7.0930 \hspace{1.1cm}
%     C_{\Z^2}^{(4)} = -2
%     % \end{aligned}
%     \\
%     &C_{\mathbf T}^{(1)} = -\frac{25}{6} - \frac{5\sqrt{3}}{2\pi} + 
%     \frac{297}{\pi^2} - \frac{594\sqrt{3}}{\pi^3} + \frac{972}{\pi^4} \approx 1.3443\\
%     &C_{\mathbf T}^{(2)} = -\frac{35}{8} + \frac{611\sqrt{3}}{4\pi} - \frac{4077}{2\pi^2} + \frac{3159\sqrt{3}}{\pi^3} - \frac{4860}{\pi^4} \approx -0.1296\\
%     &C_{\mathbf T}^{(3)} = \frac{239}{4} - \frac{537\sqrt{3}}{\pi} + \frac{5031}{\pi^2} - \frac{6696\sqrt{3}}{\pi^3} + \frac{9720}{\pi^4} \approx -0.8286\\
%     &C_{\mathbf T}^{(4)} = -\frac{599}{6} + \frac{1433\sqrt{3}}{2\pi} - \frac{5832}{\pi^2} + \frac{7074\sqrt{3}}{\pi^3} - \frac{9720}{\pi^4} \approx -0.3339\\
%     &C_{\mathbf T}^{(5)} = \frac{247}{4} - \frac{841\sqrt{3}}{2\pi} + \frac{3240}{\pi^2} - \frac{3726\sqrt{3}}{\pi^3} + \frac{4860}{\pi^4} \approx -0.0497\\
%     &C_{\mathbf T}^{(6)} = -\frac{105}{8} + \frac{363\sqrt{3}}{4\pi} - \frac{1395}{2\pi^2} + \frac{783\sqrt{3}}{\pi^3} - \frac{972}{\pi^4} \approx -0.0026\\
%     &C_{\mathbf H}^{(1)} = \frac34 %= 0.75
%     \hspace{1cm}
%     C_{\mathbf H}^{(2)} = 0 \hspace{1cm}
%     C_{\mathbf H}^{(3)} = -\frac34 %= -0.75   
% \end{align*}

\section*{Declarations}
\subsection*{Funding}
AR is funded by the grant OCENW.KLEIN.083 from the Dutch Research Council. AR acknowledges the hospitality of UCL, where part of this work was carried out.
% \appendix
% \input{appendix.tex}
% \input{symbols.tex}

% \appendix
% \input{appendix.tex}
%%%%Start biblio
\bibliographystyle{abbrvnat}
\bibliography{references}

% \printunsrtglossary[title={Notation summary}]

\end{document}